\documentclass[twoside]{article}

\usepackage[accepted]{aistats2022}
\usepackage[round]{natbib}

\usepackage[utf8]{inputenc} 
\usepackage[T1]{fontenc}    
\usepackage{hyperref}
\usepackage{url}            
\usepackage{booktabs}       
\usepackage{amsfonts}       
\usepackage{nicefrac}       
\usepackage{microtype}      
\usepackage{xcolor}         
\usepackage{amsmath}
\usepackage{cleveref}

\usepackage{mathtools}
\usepackage{algorithm,algcompatible}
\usepackage{enumitem} 
\usepackage{graphicx}
\usepackage{caption}
\usepackage{subcaption}
\newcommand{\mcs}{\mathcal{S}}
\newcommand{\mca}{\mathcal{A}}
\newcommand{\mcH}{\mathcal{H}}
\newcommand{\mcx}{\mathcal{X}}
\newcommand{\mR}{\mathbb{R}}
\newcommand{\mcL}{\mathcal{L}}
\newcommand{\mE}{\mathbb{E}}

\newcommand{\prob}{\mathrm{Pr}}
\newcommand{\cb}{c}
\newcommand{\dtv}{\textrm{d}_{TV}}
\newcommand{\nn}{\nonumber}

\DeclareMathOperator*{\argmax}{arg\,max}
\DeclareMathOperator*{\argmin}{arg\,min}
\newtheorem{thm}{Theorem}
\newtheorem{assump}{Assumption} 
\newtheorem{prop}{Proposition}
\newtheorem{corlry}{Corollary}
\newtheorem{newlemma}{Lemma}
\newtheorem{definition}{Definition}
\newcommand{\tsum}{\textstyle{\sum}}

\newenvironment{proof}{\paragraph{Proof}}{\hfill$\square$}
\crefname{assump}{assumption}{assumptions}

\begin{document}

%

%

\twocolumn[

\aistatstitle{Faster Algorithm and Sharper Analysis for Constrained Markov Decision Process}

\aistatsauthor{ Tianjiao Li$^\star$ \And Ziwei Guan$^\dagger$ \And Shaofeng Zou$^\ddagger$  \And Tengyu Xu$^\dagger$ \And Yingbin Liang$^\dagger$ \And Guanghui Lan$^\star$ }

\aistatsaddress{ $^\star$Georgia Institute of Technology \qquad$^\dagger$The Ohio State University\\ $^\ddagger$University at Buffalo, The State University of New York }   ]

\begin{abstract}
  The problem of constrained Markov decision process (CMDP) is investigated, where an agent aims to maximize the expected accumulated discounted reward subject to multiple constraints on its utilities/costs. A new primal-dual approach is proposed with a novel integration of three ingredients: entropy regularized policy optimizer, dual variable regularizer, and Nesterov's accelerated gradient descent dual optimizer, all of which are critical to achieve a faster convergence. The finite-time error bound of the proposed approach is characterized. Despite the challenge of the nonconcave objective subject to nonconcave constraints, the proposed approach is shown to converge to the global optimum with a complexity of $\tilde{\mathcal O}(1/\epsilon)$ in terms of the optimality gap and the constraint violation, which improves the complexity of the existing primal-dual approach by a factor of $\mathcal O(1/\epsilon)$ \citep{ding2020natural,paternain2019constrained}. This is the first demonstration that nonconcave CMDP problems can attain the complexity lower bound of $\mathcal O(1/\epsilon)$ for convex optimization subject to convex constraints. Our primal-dual approach and non-asymptotic analysis are agnostic to the RL optimizer used, and thus are more flexible for practical applications. More generally, our approach also serves as the first algorithm that provably accelerates constrained nonconvex optimization with zero duality gap by exploiting the geometries such as the gradient dominance condition, for which the existing acceleration  methods for constrained convex optimization are not applicable. 
\end{abstract}

\section{Introduction}
In reinforcement learning (RL), an agent interacts with a stochastic environment over time in order to maximize its expected accumulated reward \citep{sutton2018reinforcement}. In many practical applications such as autonomous driving \citep{fisac2018general} and robotics \citep{ono2015chance}, it is also critical for an agent to meet certain safety constraints, for example, an accumulated {\em utility} should be kept above a given threshold (or equivalently an accumulated {\em cost} should be kept below a given threshold). Such safe RL problems \citep{garcia2015comprehensive} are commonly formulated as constrained Markov decision processes (CMDPs) \citep{altman1999constrained}.
%
One method used widely in practice is the primal-dual approach, which constructs a Lagrangian function and then solves a minimax optimization problem respectively over dual variables and the policy.
  
Various algorithms have been proposed based on the primal-dual approach \citep{tessler2018reward,ding2020provably,stooke2020responsive,yu2019convergent,achiam2017constrained,yang2019projection,altman1999constrained}. Theoretically, two recent studies \citep{paternain2019constrained} and \citep{ding2020natural} showed that algorithms with alternating updates between policy optimization over the policy and gradient descent over dual variables achieve the computational complexity of $\mathcal{O}(1/\epsilon^2)$ to attain an $\epsilon$-accurate solution (in terms of duality gap in \citep{paternain2019constrained}, and both the optimality gap and constraint violation  in \citep{ding2020natural}).
Such a complexity result appears natural, because the CMDP problem has a {\bf nonconcave} optimization objective with {\bf nonconcave} constraints. However, recent understanding about (C)MDPs does suggest a promise for a better rate. Several recent studies \citep{mei2020global,agarwal2019theory,bhandari2020note,cen2020fast} on policy gradient (PG) approaches showed that they can attain global optimum with as fast as linear convergence rate, which suggest that MDP objective has better geometry in nature beyond nonconcavity. These optimistic insights thus motivate us to aim at a more ambitious goal of achieving the lower bound $\mathcal O( 1/\epsilon)$ on the computational complexity given by \citep{ouyang2019lower,nemirovsky1992information} for {\bf convex} optimization with {\bf convex} constraints. Therefore, the following challenging but fundamental question is asked.
\begin{list}{}{\topsep=0.5ex \leftmargin=0.3in \rightmargin=0.3in \itemsep=0.0in} 
\item {\em Despite the nonconcave objective function and nonconcave constraints, can we develop an approach for CMDP to achieve the complexity lower bound of $\mathcal O(1/\epsilon)$ for convex optimization with convex constraints?} 
\end{list}
 
We provide an affirmative answer to the above question. Specifically, we develop a novel primal-dual approach for CMDP problems, and establish its non-asymptotic global convergence guarantee with the desired complexity of $\tilde{\mathcal O}(1/\epsilon)$, where the notation $\tilde{\mathcal O}(\cdot)$ omits the dependence on logarithm terms. 

Our approach can be naturally extended more broadly to the following nonconvex functional constrained optimization problem, which includes the above CMDP problem as a special case:
\begin{align}
\min_{x\in X}\; f_0(x),\quad \text{s.t.}\;\; f_i(x) \leq 0,\quad \text{for }i=1,\ldots,m,\label{cons_opt}
\end{align}
where $X\subseteq \mathbb{R}^n$ is a nonempty closed convex set and $f_i(x), i=0,\ldots,m$ are nonconvex  functions with $|f_i(x)|\leq G,~ \forall x \in X$. Existing literature on such a {\bf nonconvex constrained} problem is rather limited \cite{boob2019stochastic,facchinei2021ghost,wang2017penalty}, and all these studies provide only the guarantee to {\em stationary points}. We generalize our approach to exploit geometric properties that nonconvex constrained optimization may have such as the gradient dominance condition, and show that our approach achieves a fast convergence rate of $\tilde{\mathcal{O}}\left({1}/{\sqrt{\epsilon}}\right)$ with guaranteed convergence to the {\em global optimum}.

\subsection{Main Contributions}
Our main contributions in this paper can be summarized as follows.

We propose a novel primal-dual approach, called Accelerated and Regularized Constrained Policy Optimization (AR-CPO), for solving constrained MDPs. Our AR-CPO approach includes a novel blend of three ingredients: (a) entropy regularized policy optimizer, (b) dual variable regularizer, and (c) Nesterov’s accelerated gradient descent dual optimizer, each of which is necessary and indispensable for attaining the improved complexity.

We show that AR-CPO attains an $\epsilon$-accurate optimal policy with $\epsilon$-level constraint violation with a computational complexity of $\tilde{\mathcal O}(1/\epsilon)$, which improves the best known complexity of $\mathcal O(1/\epsilon^2)$ for the primal-dual approaches in \cite{ding2020natural,paternain2019constrained} by a factor of $\mathcal O(1/\epsilon)$. 
Interestingly, our result achieves the complexity lower bound for convex optimization with convex constraints up to a logarithmic dependence on $\epsilon$, although CMDP has nonconcave objective and nonconcave constraints.

Our technical analysis includes the following new developments.
{\bf (a)} To prove that the entropy-regularized dual function is smooth, our analysis involves a novel application of Mitrophanov's perturbation bound \citep{mitrophanov2005sensitivity} for showing that the visitation measure is Lipschitz with respect to (w.r.t.) the policy and further technical arguments that the optimal policy for the entropy-regularized MDP is Lipschitz w.r.t.\ its reward function and hence the dual variable. 
{\bf (b)} Our analysis of Nesterov’s accelerated gradient descent dual optimizer is non-trivial and is significantly different from the Nesterov's analysis for minimization problems. 
{\bf (c)} We exploit the first-order optimality condition of the dual variable update, and establish an upper bound on the optimality gap and constraint violation using dual variable updates. 

We further extend our approach to accelerate a class of nonconvex functional constrained optimization problems with zero duality gap (see \cref{cons_opt}). By exploiting the geometric properties such as the gradient dominance condition, our approach achieves the $\mathcal{O}\left({1}/{\sqrt{\epsilon}}\right)$ rate of convergence to the global optimum, which also matches the lower bound of smooth convex optimization up to some logarithmic factor. This is the first global convergence guarantee for nonconvex constrained optimization in the literature.

\subsection{Related Works}

\textbf{Primal-dual approach to CMDP.} The primal-dual approach is a popular method to solve the CMDP problem.
Based on such an approach, various algorithms have been proposed \citep{tessler2018reward,ding2020provably,stooke2020responsive,yu2019convergent,achiam2017constrained,yang2019projection,altman1999constrained}. On the theory side, \citep{tessler2018reward} provided an asymptotic convergence analysis for primal-dual method and established a local convergence guarantee under certain stability assumptions. \citep{paternain2019constrained} showed that the primal-dual method achieves zero duality gap. \citep{ding2020provably} proposed a primal-dual type proximal policy optimization (PPO) and established the regret bound for {\em linear} CMDP. More recently, \citep{ding2020natural} provided a finite-time analysis for an alternating updating algorithm between natural policy gradient and project gradient descent. Both \citep{paternain2019constrained} and \citep{ding2020natural} showed that the existing common primal-dual framework achieves the complexity order of $\mathcal{O}(1/\epsilon^2)$. This paper proposes a new primal-dual approach, which enhances such a complexity order by a factor of $\mathcal{O}(1/\epsilon)$.

\noindent\textbf{Other approaches to CMDP.} In addition to the primal-dual approaches, other methods have also been proposed to solve the CMDP problem. Notably, \citep{liu2019ipo} developed an interior point method, which applies logarithmic barrier functions for CMDP. \citep{chow2018lyapunov,chow2019lyapunov} leveraged Lyapunov functions to handle constraints. \citep{dalal2018safe} introduced a safety layer to the policy network to enforce constraints. None of these methods is provably convergent to an optimal feasible policy. Recently, \citep{xu2020primal} proposed a new primal approach which alternatively updates between policy maximization and constraint descent (if violated). Such an algorithm has been shown to converge to the optimal policy.

\noindent\textbf{Global optimality of policy optimization.} In the unconstrained setting,  recent studies have established the global convergence of policy optimization algorithms despite the nonconvexity of the objective. For general MDPs, it has been shown that both policy gradient and natural policy gradient converge to the global optimal policy in the tabular or function approximation setting \cite{agarwal2019optimality, wang2019neural,liu2019neural,xu2020improving}. For regularized MDP, entropy regularizer has been widely used for designing RL algorithms \citep{williams1991function,peters2010relative,mnih2016asynchronous,duan2016benchmarking,haarnoja2018soft,hazan2019provably,xiao2019}. The theoretical performance of such designed algorithms have recently been characterized. The regularized MDP problem has been studied in \cite{nachum2017bridging}, in which the authors revealed the connection between the optimal value function and optimal policy, see also   
\cite{shani2020adaptive,agarwal2019theory,mei2020on,cen2020fast}. Linear convergence of policy gradient methods
for both regularized and general MDP problems has been recently established in \cite{LanPMD2021}.

\noindent\textbf{Nonconvex constrained optimization.} The studies that deal with the nonconvex functional constraints in \cref{cons_opt} appear to be scarce and only recently. In \cite{facchinei2021ghost,wang2017penalty}, the authors treat such a problem using penalty methods with first order oracles, where the convergence established is to stationary points. In \cite{boob2019stochastic} the problem is decomposed into a series of strongly convex constrained optimization subproblems by adding sequentially designed strongly convex regularizers. Convergence to stationary points was shown under certain conditions. Our study exploits the geometries that nonconvex functions can have to design an algorithm that achieves the global optimum. Note that existing analysis on convex constrained optimization cannot be applicable for establishing such a result (see Section 2 of \cite{boob2019stochastic} and references therein).

\section{Preliminaries}
\subsection{Markov Decision Process}
A Markov decision process (MDP) is determined by a five-tuple $(\mcs, \mca, \mathsf{P},r,\gamma)$, where $\mcs$ is the state space, $\mca$ is the action space, $\mathsf{P}$ is the transition kernel, $r$ is the reward function and $\gamma\in(0,1)$ is the discount factor. Assume that  $\mcs$ and  $\mca$ are finite with cardinality $|\mcs|$ and $|\mca|$, respectively. At any time $t\in \mathbb{N}_+$, an agent takes an action $a_t\in\mca$ at state $s_t\in\mcs$, and then the environment transits to the next state $s_{t+1}$ according to the distribution $ \mathsf{P}(s_{t}|s_{t-1}, a_{t-1})$. At the same time, the agent receives a reward of $r(s_t, a_t)$. We assume that the initial state $s_0$ follows a distribution $\rho$. The goal is to maximize the expected accumulated discounted reward: $\mE[\sum_{t=0}^{\infty}\gamma^tr(s_t, a_t)]$. 
A stationary policy maps a state $s\in\mcs$ to a distribution $\pi(\cdot|s)$ over $\mca$, which does not depend on time $t$. 
  
For a given policy $\pi$, we define its value function for any initial state $s\in\mcs$ as 
$
V_r^\pi(s)\coloneqq\mE[\sum_{t=0}^{\infty}\gamma^t r(s_t,a_t)|s_0=s, a_t\sim \pi(a_t|s_t),s_{t+1}\sim \mathsf{P}(\cdot|s_t,a_t)].
$
We further take the expectation with respect to the distribution of the initial state, and define the expected reward by following policy $\pi$ as $V_r^\pi(\rho) \coloneqq \mE_{s_0\sim\rho}[V_r^\pi(s_0)]$. We define the discounted state-action visitation distribution $\nu_\rho^{\pi}$ as follows:
$
    \nu_\rho^{\pi}(s,a) \coloneqq (1-\gamma)\sum_{t=0}^\infty \gamma^t\prob\big\{s_t =s, a_t= a\big| s_0 \sim \rho, a_t \sim \pi(\cdot|s_t), s_{t+1} \sim\mathsf{P}(\cdot|s_t, a_t)\big\},
$
for any $s\in\mcs$, $a\in\mca$.
The value function thus can be equivalently written as 
\begin{align*}
    V_r^\pi(\rho) =\frac{\sum_{s\in\mcs, a\in\mca}  \nu_\rho^{\pi}(s,a) r(s,a)}{1-\gamma} = \frac{\langle \nu_\rho^\pi, r\rangle_{\mcs\times\mca}}{1-\gamma},
\end{align*}
where $\langle \cdot, \cdot\rangle_{\mcs\times\mca}$ denotes the inner product over the space $\mcs\times\mca$ by reshaping $\nu_\rho^{\pi}$ and $r$ as  $|\mcs|\times|\mca|$-dimensional vectors, and we omit the subscripts and use $\langle\cdot,\cdot\rangle$ when there is no confusion.  

\subsection{Constrained MDP}\label{sec:cmdp}
The CMDP has exactly the same dynamic as the general MDP except the reward is an $(m+1)$-dimensional vector. Specifically, taking action $a\in\mca$ at state $s\in\mcs$, the agent receives a vector of rewards, ${r}(s,a) = [r_0(s,a), r_1(s,a), \ldots, r_m(s,a)]^\top$. 
For each reward function $r_i$, $i= 0, 1, \ldots, m$, we assume that it is positive and finite. We let $r_{i, \max } \coloneqq \max_{s\in\mcs, a\in\mca}\{r_i(s,a)\}$, for $i=0,1,...,m$,  and let $R_{\max} \coloneqq \sqrt{\sum_{i=1}^m r_{i,\max}^2}$.
We define the value function with respect to the $i$-th component of the reward vector ${r}$ as follows
$
V_{i}^\pi(s)\coloneqq\mE[\sum_{t=0}^{\infty}\gamma^t r_i(s_t,a_t)|s_0=s, a_t\sim \pi(a_t|s_t),s_{t+1}\sim \mathsf{P}(\cdot|s_t,a_t)],
$
for $i=0,1,...,m$. Similarly, we define $V^\pi_i(\rho)=\mE_{s_0\sim\rho}[V_i^\pi(s_0)]$.
The objective of the constrained MDP is to solve the following constrained optimization problem:
\begin{align}
    \max_{\pi\in\Pi} \quad & V_0^\pi(\rho)\nonumber \\
    s.t. \quad & V^\pi_i(\rho)\ge c_i, \quad i = 1, \ldots, m, \label{eq:constrainedrl}
\end{align}
where $\Pi=\big\{\pi\in\mR^{|\mcs||\mca|}: \sum_{a\in\mca} \pi(a|s) = 1, \pi(a|s)\geq 0,$ $\forall (s,a)\in\mcs\times\mca\big\}$ is the set of all stationary policies. 
In this paper, we focus on the case with direct parameterization of the policy. Let $\pi^*$ denote the policy that attains the optimum of the problem in \cref{eq:constrainedrl}.
We note that both the objective and constraint functions are nonconcave in $\pi$  \citep{bhandari2019global,agarwal2019theory}, which makes the problem challenging. 

Our goal is to solve the CMDP problem in \cref{eq:constrainedrl} and find an $\epsilon$-optimal policy defined as follows.
\begin{definition}\label{def:gap}
A policy $\tilde \pi$ is said to be $\epsilon$-optimal if its corresponding optimality gap and the constraint violation satisfy
\begin{align}
V_0^*(\rho)-V^{\tilde \pi}_0(\rho) \le\epsilon;  \text{and }  \|({c}-{V}^{\tilde \pi}(\rho))_+\|_1\le\epsilon,
\end{align}
where $V_0^*(\rho)$ is the optimal value of eq. \eqref{eq:constrainedrl}, ${V}^{\tilde \pi}(\rho):=[{V}_1^{\tilde \pi}(\rho),...,{V}_m^{\tilde \pi}(\rho)]^\top$ and ${c}:=[c_1,...,c_m]^\top$.
\end{definition}
We note that this is a stronger performance guarantee than the metric of the duality gap  in \citep{paternain2019constrained}.

\subsection{Notations}\label{sec:notation}
Given a vector ${x}\in\mR^m$, denote its $i$-th entry by ${x}_i$. Let $\|{ x}\|_1\coloneqq\sum_{i=1}^m |{x}_i|$, $\|{ x}\|_2 \coloneqq \sqrt{\sum_{i=1}^m{x}_i^2}$ and $\| x\|_\infty :=\max_{i} |x_i|$ denote the $\ell_1$, $\ell_2$ and $\ell_\infty$ norms respectively. Let $(\cdot)_+$ be the entry-wise rectified linear function, i.e., given $ {x}$, the $i$-th entry of $( {x})_+$ is $\max\{0, {x}_i\}$. Given two vectors $x$ and $y$, we use $\langle  x,  y\rangle$ or $ x^\top  y$ to denote the inner product. We say $ x\succeq  y$, if and only if $x_i\ge y_i$, for all $i=1,\ldots, m$. Let $\mathbf{1}$ be the $m$-dimensional all-one vector. Given two functions $f(x)$ and $g(x)$, $f(x) =\mathcal{O}(g(x))$ if there exists some constant $a>0$, such that$f(x)\le ag(x)$ for all $x$; and $f(x)=\tilde{\mathcal{O}}(g(x))$ if there exists some constant $a>0$ and $n\in\mathbb{N}$, such that $f(x)\le ag(x)\log^n(g(x))$ for all $x$. Let $[T]$ denote the set $\{1,2,...,T\}$.

\section{A Novel Primal-Dual Approach to CMDP}\label{sec:cmdp3}

\subsection{Existing Primal-Dual Approach}
\begin{algorithm}
	\caption{Existing Primal-Dual Approach}\label{alg:primaldual}
	\begin{algorithmic}[1]
	    \STATE \textbf{Input:} $\eta$.
		\STATE \textbf{Initialize:} $\lambda_0 = 0$, $\pi_0 =0$.
		\FOR {$t= 0, 1, ...,T$}
		\STATE Compute $\pi_{t+1}$ based on $\mcL(\pi, \lambda_t)$ and $\pi_t$ via a policy optimization method.
		\STATE Compute the dual ascent step $\lambda_{k+1} = (\lambda_k - \eta ( V^{\pi_{t+1}}(\rho) - c))_+$.
		\ENDFOR
	\end{algorithmic}
\end{algorithm}
A popular method to solve the CMDP problem in \cref{eq:constrainedrl} is the primal-dual approach. The idea is to solve the following minimax problem over a constructed Lagrangian function given by
\begin{align}
\min_{\lambda\in\mR^m_+}\max_{\pi\in\Pi}  \mcL(\pi, \lambda) &:= V_0^\pi(\rho)+\sum_{i=1}^m \lambda_i (V^\pi_i(\rho) -c_i) \nonumber\\&~= V_0^{\pi}(\rho) +\langle \lambda, V^\pi(\rho) -c\rangle,\label{eq:minmax}
\end{align}
where $\lambda=[\lambda_1,...,\lambda_m]^\top$ denotes the dual variable vector, $V^\pi(\rho) = [V^\pi_1(\rho), \ldots, V^\pi_m(\rho)]^\top$ denotes the vector of constraints and $c= [c_1, \ldots, c_m]^\top$ denotes the vector of constraint thresholds. 

The problem in \cref{eq:minmax} is commonly solved via a general algorithmic framework summarized in \Cref{alg:primaldual}, which alternatively updates the policy via a policy optimization RL algorithm and the dual variable $\lambda$ via a projected gradient descent. Under such a framework, existing primal-dual algorithms can have various designs. For example, in \citep{ding2020natural}, a natural policy gradient is adopted for the policy update, and in \citep{paternain2019constrained}, the RL algorithm for updating the policy is not specified, but each update needs to find an approximate solution to $\max_{\pi\in\Pi} \mcL(\pi, \lambda_t)$.

Recent theoretical studies have shown that the primal-dual approach in \Cref{alg:primaldual} achieves the convergence rate of $\mathcal{O}(1/\sqrt{T})$, which corresponds to the computational complexity of $\mathcal{O}({1}/{\epsilon^2})$ to either attain an $\epsilon$-accurate duality gap with a RL policy optimizer \citep{paternain2019constrained}, or to attain an $\epsilon$-accurate optimality gap and $\epsilon$-level constraint violation with a natural policy gradient algorithm \citep{ding2020natural}.

\subsection{AR-CPO: A New Primal-Dual Approach}\label{sec:arcpo}
\begin{algorithm}
	\caption{AR-CPO: A New Primal-Dual Approach}\label{alg:arcpo}
	\begin{algorithmic}[1]
		\STATE \textbf{Input:} $T$, stepsizes $\eta$, $\alpha$ and $q$, regularization parameters $\tau$ and $\mu$,  accuracy $\delta$, and projection radius $B$
		\STATE \textbf{Initialize:} $\lambda_{0} = 0$, $\bar \lambda_0 = 0$
		\FOR {$t= 1, ..., T$}
		\STATE $\underline \lambda_t \leftarrow (1-q)\bar\lambda_{t-1} + q\lambda_{t-1}$
		\STATE $\pi_t \leftarrow \mathrm{RegPO}(\delta, \underline\lambda_t, \tau)$
		\STATE $\widehat \nabla d_{\tau,\mu}(\underline\lambda_t) \leftarrow {V}^{\pi_t}(\rho) - c + \mu\underline\lambda_t$
		\STATE $\lambda_t \leftarrow \argmin_{\lambda \in \Lambda}\big\{\eta[\langle  \widehat\nabla d_{\tau,\mu}(\underline\lambda_t),\lambda \rangle+\frac{\mu}{2}\|\lambda - \underline \lambda_t\|_2^2] + \frac{1}{2} \|\lambda -\lambda_{t-1}\|^2_2\big\}$,  \[\qquad\textrm{where }\Lambda =\{\lambda\in\mR_+^m: \lambda_i \le 2B, i=1, \dots, m\}\]
		\STATE $\bar\lambda_t \leftarrow (1-\alpha)\bar\lambda_{t-1} + \alpha\lambda_t$
		\ENDFOR
		\STATE \textbf{Output:} $\tilde \pi$, where for any given $a\in\mca$ and $s\in\mcs$ $\tilde \pi$ is defined as 
			\begin{align*}
				\tilde{\pi} (a|s) = \tfrac{(1-\alpha)^{T-1} \nu_\rho^{\pi_1}(s,a) + \sum\nolimits_{t=2}^T \alpha(1-\alpha)^{T-t} \nu_\rho^{\pi_{t}}(s,a)}{(1-\alpha)^{T-1} \chi^{\pi_1}_\rho(s) + \sum\nolimits_{t=2}^T \alpha(1-\alpha)^{T-t} \chi_\rho^{\pi_{t}}(s)},
			\end{align*}
		where $\chi_\rho^\pi (s) = \tsum_{a^\prime\in\mca} \nu_\rho^\pi(s, a^\prime)$.
	\end{algorithmic}
\end{algorithm}

One prominent drawback of the existing primal-dual approach in \Cref{alg:primaldual} is that the overall complexity is dominated by the minimization over the dual function, namely $\min_{\lambda} d(\lambda):=\max_{\pi\in\Pi}  \mcL(\pi, \lambda).$
Since $d(\lambda)$ is nonsmooth due to the maximization over $\pi$, the (projected) gradient descent over $\lambda$ decays only at the rate of $\mathcal{O}(1/\sqrt{T})$, although the RL policy optimizer typically achieves a much faster rate of $\mathcal{O}(1/T)$. Hence, the overall  complexity  is only $\mathcal{O}({1}/{\epsilon^2})$. 

In this paper, we propose a novel primal-dual approach, which we call as Accelerated and Regularized Constrained Policy Optimizer (AR-CPO), and is summarized in \Cref{alg:arcpo}. The central idea is to solve the minimax problem over the $(\tau, \mu)$-regularized Lagrangian via an accelerated dual descent:
\begin{align}\label{eq:minmaxreg}
\min_{\lambda\in\mR^m_+}\max_{\pi\in\Pi} \mcL_{\tau, \mu}(\pi, \lambda) \coloneqq \mcL(\pi, \lambda)+\tau\mathcal H(\pi)+\frac{\mu}{2}\|\lambda\|_2^2, 
\end{align}
where $ \mcH(\pi) = -\mE\big[\sum_{t=0}^{\infty}\gamma^t\log(\pi(a_t|s_t)) \big|s_0=s, a_t\sim \pi(a_t|s_t),s_{t+1}\sim \mathsf{P}(\cdot|s_t,a_t)\big]$ is the discounted entropy of the policy $\pi$, and $\tau,\mu\geq 0$ are regularization constants.  

Compared to the existing primal-dual approach in \Cref{alg:primaldual}, AR-CPO consists of three new components, which collectively yield the improved complexity of $\mathcal{O}({1}/{\epsilon})$.

\textbf{(a) Entropy-regularized policy optimizer.} The entropy regularizer $\tau\mcH(\pi)$ in \cref{eq:minmaxreg} plays a critical role to smooth the dual function defined as $d_{\tau, \mu}(\lambda) := \max_{\pi\in\Pi} \mcL_{\tau, \mu}(\pi, \lambda)$, which is shown in \Cref{prop:smoothness} in \Cref{proof:thm1}. Thus, the dual problem $\min_{\lambda\in\mR^m_+} d_{\tau, \mu}(\lambda)$ becomes minimization over a smooth objective, which (together with items (b) and (c) to be introduced) can enjoy a better convergence rate and an enhanced complexity order.
 
To obtain the dual function, i.e., to solve $\max_{\pi\in\Pi} \mcL_{\tau, \mu}(\pi, \lambda)$ for any given $\lambda$, it can be observed that such a problem is equivalent to the maximization of an entropy-regularized value function corresponding to a reward function $r_\lambda\coloneqq r_0 +\sum_{i=1}^m \lambda_i r_i$, which is a linear combination of the rewards. Therefore, many entropy-regularized policy optimizers can be applied here, and in \Cref{alg:arcpo}, $\text{RegPO}(\delta, \lambda, \tau)$ represents such a generic optimizer that learns an optimal policy for the entropy-regularized MDP with reward $r_\lambda$ up to $\delta$-accuracy. We give two example RegPOs in \Cref{app:regpo}: RegPO-NPG based on natural policy gradient method and RegPO-SoftQ based on soft Q-learning.  

In this paper, we focus on the deterministic case and assume that RegPO has the access to an MDP oracle that provides the expected value for an update, such as exact policy gradient in RegPO-NPG or soft Bellman operator in RegPO-SoftQ. Thus, the computational complexity is measured by the number of queries of such an oracle. Due to the regularization of MDP, each call of the optimizer RegPO $(\delta,\lambda,\tau)$ requires  
$\mathcal{O}(\log(1/(\delta\tau))) = \mathcal{O}(\log({1}/{\epsilon}))$
queries of the MDP oracle (see \Cref{app:regpo} for more details). 

\textbf{(b) $\ell_2$ regularization on $\lambda$.} We further introduce an $\ell_2$-regularizer $\frac{\mu}{2}\|\lambda\|_2^2$ on the dual variable $\lambda$ in \cref{eq:minmaxreg} so that the dual function $d_{\tau, \mu}(\lambda)$ becomes a strongly convex function with respect to $\lambda$. This design, together with the following Nesterov's accelerated gradient method for updating $\lambda$, will yield the enhanced complexity order.

\textbf{(c) Nesterov’s accelerated gradient descent dual optimizer.} We further adopt a variant of Nesterov’s accelerated gradient descent method, which is discussed in \cite[Section 3.3]{lan2020first}, for the update of the dual variable $\lambda$. In contrast to the gradient descent method, such an acceleration method builds up a lower approximation $d_{\tau,\mu}(\underline\lambda_t)+\langle \nabla d_{\tau,\mu}(\underline\lambda_t), \lambda - \underline\lambda_t \rangle + \tfrac{\mu}{2}\|\underline\lambda_t-\lambda\|_2^2$ of the objective function at the search point ${\underline \lambda_t}$, which is a weighted average of iterates $\{\lambda_t\}$. This approach will improve the dependence on the condition number so that the convergence rate is $\left(1-\sqrt{\mu/L_d}\right)^T$, where $L_d$ is the smoothness parameter of the dual function (will be formally defined later in \Cref{thm1}). As will be shown in \Cref{sec:result}, if we choose the regularization constant as $\mu,\tau=\mathcal{O}(\epsilon)$ so that $L_d=\mathcal{O}({1}/{\epsilon})$, then $T=\tilde{\mathcal{O}}({1}/{\epsilon})$ iterations yield an $\epsilon$-accurate optimality gap and $\epsilon$-level violation. Thus, we achieve 
an overall complexity of $\tilde{\mathcal{O}}({1}/{\epsilon})$, which improves the existing primal-dual approach by a factor of $\tilde{\mathcal{O}}({1}/{\epsilon})$. Note that without the acceleration, the vanilla gradient decent update for $\lambda$ does not yield the improved complexity order.
 

\section{Convergence Rate and Optimality}\label{sec:result}
In this section, we  present the global convergence and complexity of Algorithm \ref{alg:arcpo}.

\subsection{Main Results}\label{sec:mainresult}

We first state our technical assumptions that are commonly used in the literature of optimization and reinforcement learning theory.
\begin{assump}[Slater Condition]\label{assumption:slatercondition}
There exists a constant $\xi\in\mR_+$, and at least one policy $\pi_\xi\in\Pi$, such that for all $i= 1,\ldots , m,$ $V_{i}^{\pi_\xi}\ge c_i +\xi$.
\end{assump}
The Slater condition is easy to satisfy in practice, as we usually has one policy that is strictly feasible, i.e., all the constraints are satisfied strictly.
 
\begin{assump}[Uniform Ergodicity]\label{assumption:ergodicity} For any $ \lambda\in\Lambda$, the Markov chain induced by the policy $\pi^*_{\tau,\lambda}$ and the Markov transition kernel is  uniformly ergodic, i.e., 
there exist constants $C_M>0$ and $0<\beta<1$ such that for all $t\ge 0$, 	
	\begin{equation*}
	\sup_{s\in \mcs}\dtv\left(\mathbb{P}(s_t\in\cdot|s_0 =s), \chi_{\pi^*_{\tau,\lambda}}\right) \le C_M \beta^{t},
	\end{equation*}
	where  $\pi^*_{\tau,\lambda}$ is the solution to $\max_{\pi\in\Pi} \mcL_{\tau, \mu}(\pi, \lambda)$, $\chi_{\pi^*_{\tau,\lambda}}$ is the stationary distribution of the MDP induced by policy $\pi^*_{\tau,\lambda}$, and $\dtv\left(\cdot, \cdot\right)$ is the total variation distance. 
\end{assump}
Note that a Markov chain is uniformly ergodic if it is irreducible (i.e., possibly gets to any state from any state) and aperiodic \citep{levin2017markov}. As for any $\lambda$, $\pi^*_{\tau,\lambda}$ is the optimal policy for the entropy-regularized MDP with the reward $r_\lambda$, and is therefore a softmax policy \citep{nachum2017bridging}, which takes any action with non-zero probability. Therefore, Assumption \ref{assumption:ergodicity} can be easily satisfied in practice. 

The following theorem characterizes the global convergence of Algorithm \ref{alg:arcpo} in terms of the optimality gap $V_0^*(\rho)-V_0^{\tilde{\pi}}(\rho)$ and the constraint violation $({c}-{V}^{\tilde{\pi}}(\rho))_+$. The complete proof is provided in Appendix \ref{proof:thm1}. 
	\begin{thm}\label{thm1}
Consider AR-CPO in \Cref{alg:arcpo}. Suppose \Cref{assumption:slatercondition,assumption:ergodicity} hold and we choose
	    \begin{align}\label{stepsize}
		\alpha = \sqrt{\tfrac{\mu}{2L_d}},~q=\tfrac{2\alpha-\mu/L_d}{2-\mu/L_d}, ~\text{and}~\eta=\tfrac{\alpha}{\mu (1-\alpha)}, 
		\end{align}
where $L_d:= \frac{2R_{\max}^2L_\nu}{(1-\gamma)^2\tau} +\mu$ is the smoothness parameter of the dual function, and $L_{\nu}\coloneqq \lceil \log_{\beta}(C_M^{-1})\rceil + (1-\beta)^{-1}+1$. Then we have the following convergence guarantee of the optimality gap and the constraint violation: 
		\begin{align*}
&\quad V_0^*(\rho)-V^{\tilde \pi}_0(\rho) \\
&\leq  \zeta_T\left(\tfrac{2R_{\max}}{1-\gamma} + \tfrac{2\sqrt{m}B}{\eta}\right)+ \sqrt{\tfrac{2L_d}{\mu}}\max_{t\in[T]} |\Delta_{t}|\\
		& \quad+\tfrac{\tau\log |\mca|}{1-\gamma}  +6\mu m B^2,  \tag{\em optimality gap}
		\end{align*}
		\begin{align*}
		&\quad\|({c}-{V}^{\tilde \pi}(\rho))_+\|_1 \\
		&\leq \zeta_T \left(\tfrac{2R_{\max}}{(1-\gamma)B} + \tfrac{2\sqrt{m}}{\eta}\right)+ \sqrt{\tfrac{2L_d}{\mu}}\tfrac{\max_{t\in[T] } |\Delta_{t}|}{B}\\
		&\quad +\tfrac{\tau \log |\mca|}{(1-\gamma)B} +6\mu m B + 2\left(1-\sqrt{\tfrac{\mu}{2L_d}}\right)^{T-1}\tfrac{\sqrt{m}R_{\max}}{1-\gamma},\\
		& \tag{\em constraint violation} 
		\end{align*}
		where
		\begin{align*}
		\zeta_T :&= 2\left(1-\sqrt{\tfrac{\mu}{2L_d}}\right)^{\frac{T}{2}} \sqrt{\eta K_0(\lambda^*_{\tau,\mu})}\\ &\quad + \sqrt{8\eta\sqrt{m}B \max_{t\in[T] }\|\delta_t\|_2},
		\end{align*}
		$\lambda_{\tau,\mu}^*:=\argmin_{\lambda\in\mR^m_+} d_{\tau, \mu}(\lambda)$, $\Delta_t := V_0^{\pi^*_{\tau,\underline \lambda_{t}}}(\rho)  + \underline \lambda_{t}^\top({V}^{\pi^*_{\tau,\underline \lambda_{t}}}(\rho) - c) - [V_0^{\pi_{t}}(\rho) + \underline \lambda_{t}^\top ({V}^{\pi_t}(\rho) - c)]$, $B=\tfrac{r_{0,\max}}{(1-\gamma)\xi}$,  $K_0(\lambda) \coloneqq d_{\tau, \mu}(\bar \lambda_0)-d_{\tau, \mu}(\lambda)+\tfrac{\alpha}{2}(\mu+\tfrac{1}{\eta})\|\lambda_0-\lambda\|_2^2$, and $\delta_t \coloneqq\widehat \nabla d_{\tau,\mu}(\underline \lambda_t)-\nabla d_{\tau,\mu}(\underline \lambda_t)$.
	\end{thm}

\Cref{thm1} indicates that the optimality gap is bounded by an exponential decaying term with the convergence rate $\left(1-\sqrt{{\mu}/({2L_d})}\right)^T$. This is a faster convergence due to the acceleration of the dual variable than the convergence of a simple gradient descent, whose convergence rate is $\left(1-{\mu}/{L_d}\right)^T$. Furthermore, the parameter $L_d$ of the dual function plays an important role here, because a smoother dual function with a smaller $L_d$ yields a faster convergence. The bound also contains other terms captured by the regularization parameters $\tau$ and $\mu$, and the convergence errors $\max|\Delta_t|$ and $\max\|\delta_t\|_2$ of the entropy-regularized policy optimizer RegPO. All these terms can be controlled to be below any required accuracy $\epsilon$. The constraint violation is also bounded via the terms that have the same nature. 

In the above theorem, the optimality gap and constraint violation are bounded in terms of the regularization parameters, the number of iterations $T$, the convergence errors of RegPO, which can be jointly chosen properly to guarantee the required $\epsilon$-accurate optimality gap and $\epsilon$-level constraint violation. In the following corollary, we present the overall complexity to converge to such an $\epsilon$-optimal policy. Its proof is given in \Cref{sec:proofofcorollary1}.
\begin{corlry}\label{corollary:totalcomplexity}
With a total computational complexity of $\tilde{\mathcal{O}}({1}/[{(1-\gamma)^{3.5}\epsilon}])$, Algorithm \ref{alg:arcpo} outputs an $\epsilon$-optimal policy with respect to both the optimality gap and constraint violation. 
\end{corlry}

\section{Generalization to Constrained Nonconvex Optimization with zero duality gap}\label{sec:generalsetting}
In this section, we extend our approach to accelerate a class of generic 
constrained nonconvex optimization problems. Our approach is the first that can accelerate convergence of constrained nonconvex optimization to global optimum under certain geometries. 

Recall the problem defined in \cref{cons_opt}. Consider the dual function defined below:
\begin{align*}
d(\lambda) := \min_{x\in X} \mcL(x, \lambda),\quad \text{and} \quad d^*:=\max_{\lambda \in \mR_+^m}d(\lambda),
\end{align*}
where $\mcL(x, \lambda) := f_0(x)+\sum_{i=1}^m \lambda_i f_i(x)$ is the Lagrangian. Let $\lambda:=[\lambda_1,...,\lambda_m]^\top$ denote the dual variable vector, and let $f(x):=[f_1(x),...,f_m(x)]^\top$ denote the constraints. 

As suggested in our AR-CPO algorithm, we further introduce an $\ell_2$ perturbation denoted by $d_\mu(\lambda) := d(\lambda) - \tfrac{\mu}{2}\|\lambda\|^2$ and $\lambda_\mu^* := \argmax_{\lambda \in \mR_+^m}d_\mu(\lambda)$. We consider a class of constrained nonconvex optimization problem satisfying the following assumptions on the geometry.

\begin{assump}[Strong Duality]\label{assum:strongdual}
Strong duality holds for \cref{cons_opt}, i.e. $f_0^*=d^*$.
\end{assump}

\begin{assump}[Dual Function Smoothness]\label{assum:dualsmooth}
	$d(\lambda)$ is $L$-smooth: $~\forall\lambda, \lambda^\prime\in\mR_+^m,$ 
	\begin{align*}
	\|\nabla d(\lambda) - \nabla d(\lambda^\prime) \|_2 = \|f(x_\lambda^*)-f(x_{\lambda'}^*) \|_2 \le L \|\lambda - \lambda^\prime\|_2,
	\end{align*}
	where $x_\lambda^* :=\argmin_{x\in X} \mcL(x, \lambda)$ and $x_{\lambda'}^* :=\argmin_{x\in X} \mcL(x, \lambda')$. Consequently, $d_\mu(\lambda)$ is $\tilde L$-smooth, where $\tilde L:=L+\mu$.
\end{assump}
\Cref{assum:dualsmooth} requires that the Lagrangian has a unique solution in terms of $x$ when $\lambda$ is fixed. For a general nonconvex problem, \Cref{assum:dualsmooth} can be guaranteed if the Lagrangian satisfies the gradient dominance condition.

\begin{assump}[Slater's Condition]\label{assum:slatercondition}
	There exists a constant $\xi\in\mR_+$, and at least one  $x_\xi\in X$, such that for all $i= 1,\ldots , m,$ $f_{i}(x_\xi)\le  -\xi$.
\end{assump}

Given the assumptions above, we propose the Accelerated and Regularized Constrained Optimizer (AR-CO), which is an extension of AR-CPO to general constrained nonconvex optimization problems. The AR-CO shares the same structure as the AR-CPO in \Cref{alg:arcpo}, and the difference lies in the update of nonconvex variable $x$ and how to generate a final output when given $x_1, \ldots, x_T$. Specifically, the AR-CO algorithm replaces the steps $5$, $6$ and $7$ in \Cref{alg:arcpo} with the following steps $5^\prime$, $6^\prime$ and $7^\prime$.
\begin{align*}
    x_t &\leftarrow \mathrm{LagOPT}(\delta, \underline\lambda_t)\tag{step $5^\prime$},
\end{align*}
where LagOPT denotes a nonconvex optimizer that can solve the Lagrangian to any prescribed accuracy $\delta>0$ under parameter $\underline\lambda_t$, which can be guaranteed if the Lagrangian satisfies the gradient dominance condition.
\begin{align*}
    \widehat \nabla d(\underline\lambda_t) \leftarrow f(x_t)  - \mu\underline\lambda_t \tag{step $6^\prime$}.
\end{align*}
The step $7^\prime$ is the same as step $7$ except replacing $\widehat \nabla d(\underline\lambda_t)$ with $-\widehat \nabla d(\underline\lambda_t)$ instead, since the step is ascent here. And also $B\coloneqq {G}/{\xi}$. For the output, we randomly pick $\tilde{x} =x_1$ with probability $(1-\alpha)^{T-1}$ and $\tilde{x}=x_i$ with probability $ \alpha(1-\alpha)^{T-i}$ for $i=2, \ldots, T$.

The following theorem characterizes the global convergence of AR-CO. 
\begin{thm}\label{thm2}
Consider the AR-CO algorithm. Suppose \Cref{assum:dualsmooth,assum:strongdual,assum:slatercondition} hold and we choose $\alpha = \sqrt{\tfrac{\mu}{2\tilde L}}$, $q=\tfrac{2\alpha-\mu/\tilde L}{2-\mu/\tilde L}$, and $\eta=\tfrac{\alpha}{\mu (1-\alpha)}$.
Then we have the following convergence guarantee of the optimality gap and the constraint violation: 
		\begin{align*}
		\mathbb{E}[f_0(\tilde x)]-f_0^* \leq~& \tilde \zeta_T \left(2\sqrt{m} G + \tfrac{2\sqrt{m}G}{\eta \xi}\right)\\
		&~+ \sqrt{\tfrac{2\tilde L}{\mu}}\max_{t\in[T]} |\Delta_{t}|  + \tfrac{6\mu mG^2}{\xi^2},  \tag{\em optimality gap}
		\end{align*}
		\begin{align*}
		\|\mathbb{E}[f(\tilde x)]_+\|_1 \leq&~ \tilde \zeta_T \left(2\sqrt{m} \xi + \tfrac{2\sqrt{m}}{\eta}\right)+ \sqrt{\tfrac{2\tilde L}{\mu}}\tfrac{\max_{t\in[T]} |\Delta_{t}|}{B}\\
		& ~ +\tfrac{6\mu mG}{\xi} + 2m G\left(1-\sqrt{\tfrac{\mu}{2\tilde L}}\right)^{T-1},
		 \tag{\em constraint violation} 
		\end{align*}
		where 
		\begin{align*}
		\tilde \zeta_T:&=2\left(1-\sqrt{\tfrac{\mu}{2\tilde L}}\right)^{\frac{T}{2}} \sqrt{\eta K_0(\lambda^*_{\mu})}\\&\quad + \sqrt{\tfrac{8\eta\sqrt{m}G}{\xi} \max_{t\in[T]}\|\delta_t\|_2},
		\end{align*}
		$\Delta_t := d({\underline \lambda_t})  - f_0(x_t) - \langle \underline \lambda_t, f(x_t)  \rangle $,  $K_0(\lambda) \coloneqq d_{ \mu}(\lambda)-d_{ \mu}(\bar \lambda_0)+\tfrac{\alpha}{2}(\mu+\tfrac{1}{\eta})\|\lambda_0-\lambda\|_2^2$, and $\delta_t \coloneqq\widehat \nabla d_{\mu}(\underline \lambda_t)-\nabla d_{\mu}(\underline \lambda_t).$
	\end{thm}
\begin{corlry}\label{corollary:totalcomplexity_ARCO}
With a computational complexity of $\tilde{\mathcal{O}}({L}/{\sqrt{\epsilon}})$, Algorithm AR-CO outputs an $\epsilon$-optimal policy with respect to both the optimality gap and constraint violation.
\end{corlry}

The properties of the constrained MDP that enable the proof of \Cref{thm1} have already been captured in \Cref{assum:strongdual,assum:dualsmooth,assum:slatercondition}. Thus, the proof of \Cref{thm2} and \Cref{corollary:totalcomplexity_ARCO} follows by replacing $V_0(\rho)$ and $V(\rho)-c$ with $-f_0(x)$ and $-f(x)$ in the proof of \Cref{thm1} and applying \Cref{assum:strongdual,assum:dualsmooth,assum:slatercondition} when needed.

\section{Experiments}	
Our experiments are conducted on the environment Acrobot-v1, OpenAI Gym \citep{mei2020on}. The Acrobot-v1 simulates the acrobot system where there exist two joints and two links. The upper joint is fixed and the lower joint is actuated by the agent. The goal is to swing the end of the lower link up to a given height. Moreover, two safe penalties are implemented when (I) the first link swings in a prohibited direction, or when (II) the second link swings in a prohibited direction with respect to the first link. 

We compare our AR-CPO algorithm with the state-of-the-art primal-dual optimization (PDO) method for CMDP \citep{chow2017risk}. In order for a fair comparison, the same neural softmax policy and the trust region policy optimization \citep{schulman2015trust} are used in both algorithms. In \Cref{fig:fig}, we show the performance of the two algorithms under their best tuned parameters. The opaque lines provides the accumulated discounted reward (for the objective and the constraints) averaged over $10$ random initialized seeds and translucent error bands have the width of two standard deviations of the $10$ random initializations, respectively. More details on the environment and parameters setting can be found in \Cref{sec:appendixnumericalsim}. The dashed lines in \Cref{fig:fig} (b) and (c) are the constraint thresholds, i.e. $c_{1}$ and $c_{2}$.

\begin{figure}[h!]
\begin{subfigure}{.25\textwidth}
  \centering
  \includegraphics[width=\linewidth]{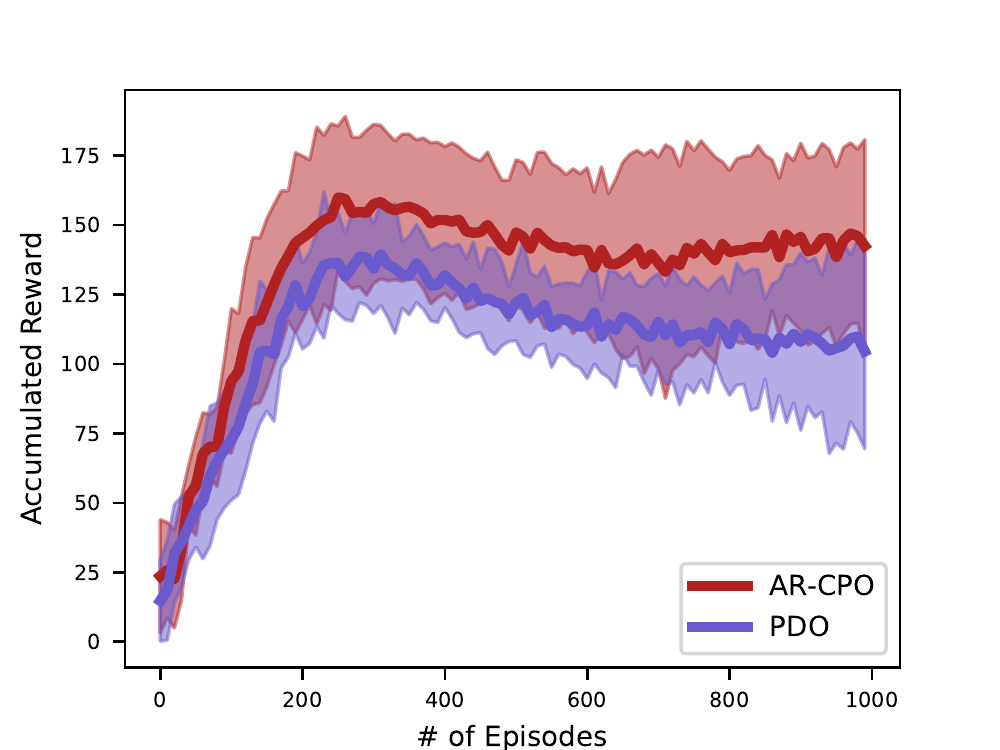}
  \caption{Objective}
  \label{fig:sfig1}
\end{subfigure}%
\begin{subfigure}{.25\textwidth}
  \centering
  \includegraphics[width=\linewidth]{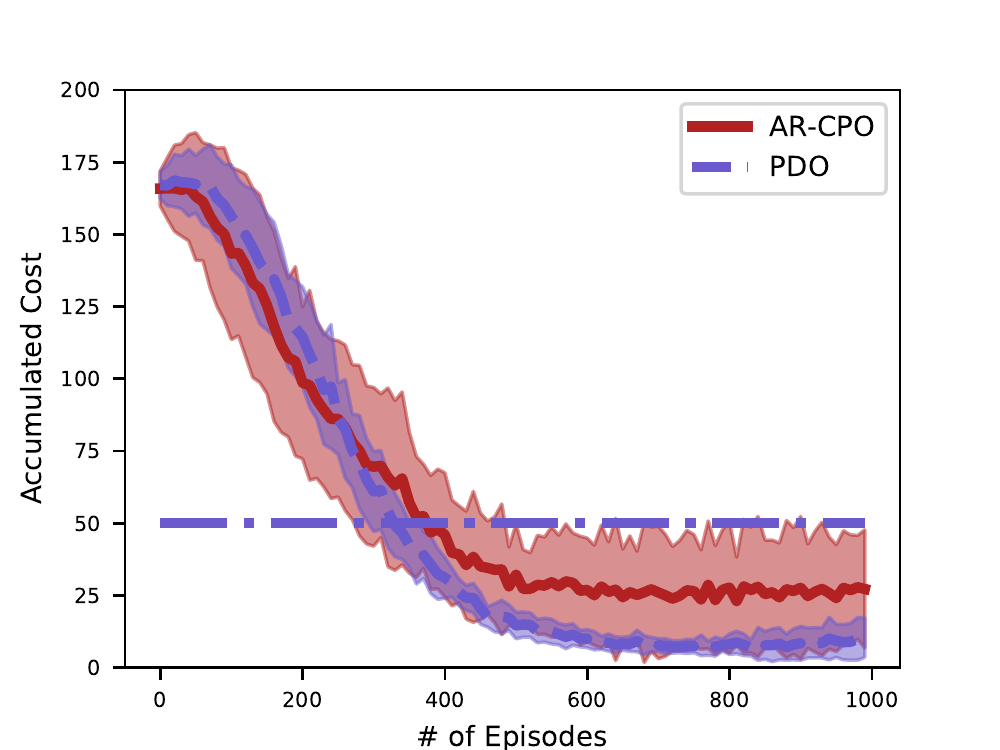}
  \caption{Constraint I}
  \label{fig:sfig2}
\end{subfigure}
\begin{subfigure}{.25\textwidth}
  \centering
  \includegraphics[width=\linewidth]{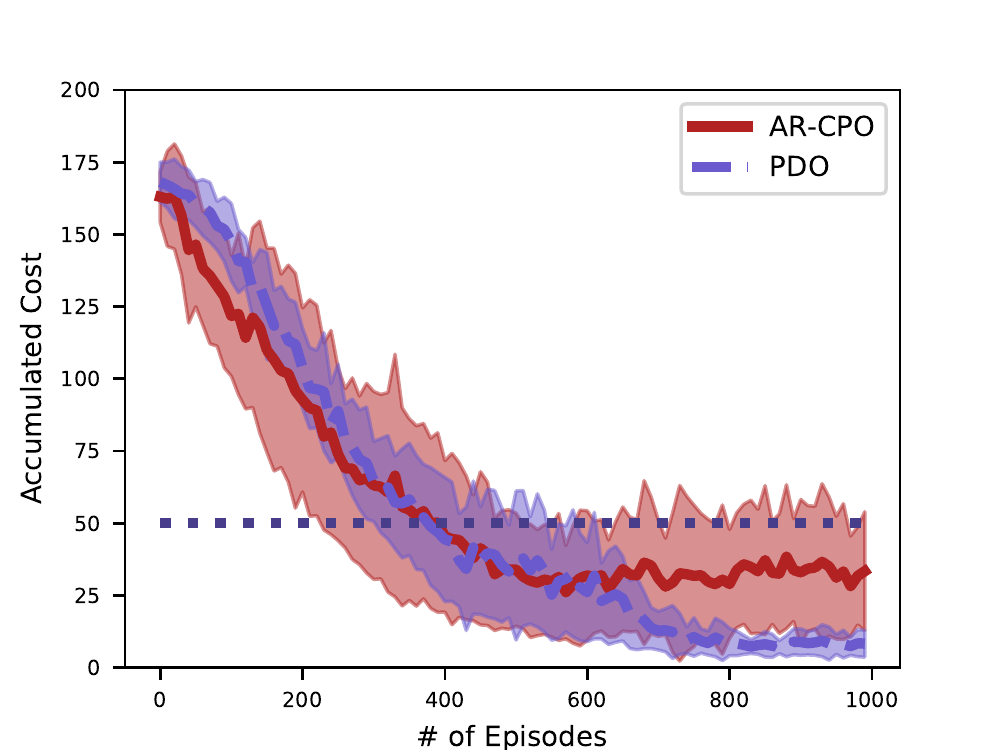}
  \caption{Constraint II}
  \label{fig:sfig3}
\end{subfigure}
\caption{Comparison between AR-CPO and PDO}
\label{fig:fig}
\end{figure}

From \Cref{fig:sfig1}, we observe that, our AR-CPO has a much higher accumulated reward in  the objective function, and the standard deviations of both algorithms are similar. \Cref{fig:sfig2,fig:sfig3} show that the accumulated reward in the constraints of AR-CPO decays slightly faster than that of PDO. After the constraints are satisfied, the cumulative reward in the constraints of AR-CPO is much closer to the threshold, which enables AR-CPO to explore  a higher reward in the objective function.

\section{Conclusion}
In this paper, we developed a novel AR-CPO approach for solving the CMDP problem, and showed that AR-CPO achieves an improved complexity of $\tilde{\mathcal {O}}(1/\epsilon)$, which improves the complexity results in  existing studies \cite{ding2020natural,paternain2019constrained} by a factor of $\mathcal O(1/\epsilon)$. Our AR-CPO approach involves a novel blend of three ingredients: (a) entropy regularized policy optimization, (b) dual variable regularizer, and (c) Nesterov’s accelerated gradient descent dual optimizer. To the best of our knowledge, this is the first establishment that CMDP, despite its nonconcave nature of the objective and constraint functions, can achieve the same complexity as the convex optimization with convex constraints. Furthermore, we extended our approach to accelerate a class of nonconvex functional constrained optimization problems, and showed that by exploiting geometric properties, e.g., the gradient dominance condition, our approach achieves the $\mathcal O(1/\sqrt{\epsilon})$ rate of convergence to the global optimum. The numerical simulations also confirm the advantage of the proposed approach. 
\bibliographystyle{authordate1}
\bibliography{rl}





\newpage
\onecolumn
\appendix
{\Large\textbf{Supplementary Materials}}
\section{Example Entropy-Regularized Policy Optimizers (RegPOs)}\label{app:regpo}

In this section, we introduce two entropy-regularized policy optimizers (RegPOs), which solves the maximization problem of an entropy-regularized value function as
\begin{align}\label{eq:entropyreg}
\max_{\pi\in\Pi} V_{\tau, \lambda}^{\pi}(\rho):=V^\pi_{r_\lambda}(\rho)+\tau \mcH(\pi)
\end{align}
where $V^\pi_{r_\lambda}$ is the value function corresponding to the reward function $r_\lambda\coloneqq r_0 +\sum_{i=1}^m \lambda_i r_i$. Note that the above problem is equivalent to the entropy regularized Lagrangian maximization problem 
$\max_{\pi\in\Pi} \mcL_{\tau, \mu}(\pi, \lambda)$ for any $\lambda\in\Lambda$. 
Define the Q-function under regularized MDP as, for all $s\in\mcs$ and $a\in\mca$, 
\begin{align*}
    Q_{\tau, \lambda}^\pi(s,a) = r_\lambda(s,a) + \gamma\mE_{s^\prime\sim\mathsf{P}(\cdot|s,a)}\left[V^\pi_{\tau,\lambda}(s^\prime)\right].
\end{align*}
And we use the short-hand notation $Q^*_{\tau, \lambda}$ to denote $Q^{\pi^*_{\tau,\lambda}}_{\tau, \lambda}$ in the Appendix.


\subsection{RegPO-NPG}

\begin{algorithm}
	\caption{RegPO-NPG $(\delta,\lambda,\tau)$ \citep{kakade2001natural,cen2020fast}}\label{alg:regpo_npg}
	\begin{algorithmic}[1]
		\STATE \textbf{Input:} Output accuracy $\delta$, dual variable value $\lambda$, regularization coefficient $\tau$, rewards $r$.
		\STATE \textbf{Initialize:} $\theta_0=0$, $K = \left\lceil\log\left(\frac{2(r_{0,\max} +BR_{\max})}{(1-\gamma)\delta\tau}\right)/\left(\log\left(\tfrac{1}{\gamma}\right)\right)\right\rceil$.
		\FOR {$k= 0, ..., K$}
		\STATE  $\theta_{k+1} \leftarrow \theta_k +  \frac{1-\gamma}{\tau}(\mathcal{F}_\rho^{\theta_k})^\dagger\nabla_\theta V_{\tau, \lambda}^{\pi_{\theta_k}}(\rho)$
		\ENDFOR
	\end{algorithmic}
\end{algorithm}

We summarize an NPG-based RegPO in \Cref{alg:regpo_npg}. Suppose a softmax parameterization of the policy $\pi$ is used. Then, the natural policy gradient (NPG) can be applied to solve the entropy-regularized MDP \citep{kakade2001natural,cen2020fast}. Specifically, the parameter update is given by 
\begin{align*}
    \theta_{k+1} = \theta_k + \beta(\mathcal{F}_\rho^{\theta_k})^\dagger\nabla_\theta V_{\tau, \lambda}^{\pi_{\theta_k}}(\rho),
\end{align*}
where $\beta$ is the stepsize, $\nabla_\theta V_{\tau, \lambda}^{\pi_{\theta_k}}(\rho)$ denotes the policy gradient, and $\mathcal{F}_\rho^{\theta_k}$ denotes the Fisher information matrix defined blow
\begin{align*}
    \mathcal{F}_\rho^{\theta_k} \coloneqq \mE_{(s,a)\sim \nu_\rho^{\pi_{\theta_k}}} \bigg[(\nabla \log\pi_{\theta}(a|
    s))(\nabla \log\pi_{\theta}(a|
    s))^\top\bigg].
\end{align*}
Let $\pi_{\tau,\lambda}^*$ denote the optimal policy for \cref{eq:entropyreg}.
It has been shown in \citep{cen2020fast} that if we take the stepsize $\beta =(1-\gamma)/\tau$, RegPO-NPG converges linearly as given by
\begin{align*}
 &   \|Q_{\tau,\lambda}^* - Q_{\tau, \lambda}^{\pi_{\theta_{K}}} \|_\infty\le \|Q_{\tau,\lambda}^* - Q_{\tau, \lambda}^{\pi_{\theta_{0}}}\|_\infty \cdot \gamma^K, \\
&    \|\pi_{K+1} - \pi_{\tau,\lambda}^*\|_\infty \le \tfrac{2}{\tau} \|Q_{\tau,\lambda}^* - Q_{\tau, \lambda}^{\pi_{\theta_{K}}} \|_\infty.
\end{align*}
Therefore, choosing $\beta =(1-\gamma)/\tau$ and $K = \left\lceil\log\left(\frac{2(r_{0,\max} +BR_{\max})}{(1-\gamma)\delta\tau}\right)/\left(\log\left(\tfrac{1}{\gamma}\right)\right)\right\rceil$ guarantees that $\|\pi_{K+1} - \pi_{\tau,\lambda}^*\|_\infty\le \delta$.

\subsection{RegPO-SoftQ}

\begin{algorithm}
	\caption{RegPO-SoftQ $(\delta,\lambda,\tau)$}\label{alg:regpo_softq}
	\begin{algorithmic}[1]
		\STATE \textbf{Input:} Output accuracy $\delta$, dual variable value $\lambda$, regularization coefficient $\tau$.
		\STATE \textbf{Initialize:} $Q_0=0$, $K =\left\lceil\log\left(\frac{2(r_{0,\max} +BR_{\max})}{(1-\gamma)\delta\tau}\right)/\left(\log\left(\tfrac{1}{\gamma}\right)\right)\right\rceil$.
		\FOR {$k= 0, ..., K-1$}
		\STATE  $Q_{k+1}\leftarrow \mathcal{T}_{\tau,\lambda}(Q_k)$
		\ENDFOR
		\STATE $\pi_{K}(a|s) \leftarrow \exp(Q_{K}(s,a)/\tau)/\left(\sum_{a^\prime\in\mca}\exp(Q_{K}(s,a^\prime)/\tau)\right)$, for all $s\in\mcs$ and $a\in\mca$.
	\end{algorithmic}
\end{algorithm}

We propose a new soft Q-learning based RegPO, which is summarized in \Cref{alg:regpo_softq}.
%
For the $\tau$-regularized MDP problem in \cref{eq:entropyreg}, 
we define the following soft Bellman operator
\begin{align*}
    (\mathcal{T}_\tau(Q))_{s,a} = r_{\lambda}(s,a) +\gamma \mE_{s^\prime\sim\mathsf{P}(\cdot|s,a)}\left[\tau\log\left(\textstyle \sum_{a^\prime\in\mca} \exp\left(\frac{Q(s^\prime, a^\prime)}{\tau}\right)\right)\right],
\end{align*}
for every $s\in\mcs$ and $a\in\mca$, where we adopt the soft Q-learning update \citep{nachum2017bridging,cen2020fast} for the Q-function rather than the ``max" operator.

It can be shown that  the Q-function converges linearly:
\begin{align*}
    \|Q_K - Q^*_{\tau, \lambda}\|_\infty \le \|Q_0 -Q^*_{\tau, \lambda}\|_\infty\cdot \gamma^K.
\end{align*}
Then the relationship $\pi_{K}(a|s)\propto\exp(Q_K(s,a)/\tau)$ further implies that
\begin{align*}
    \|\pi_{K}- \pi^*_{\tau, \lambda}\|_\infty \le \tfrac{2}{\tau}\|Q_K - Q^*_{\tau, \lambda}\|_\infty,
\end{align*}
which yields $\|\pi_{K} - \pi_{\tau,\lambda}^*\|_\infty\le \delta$ for 
$K = \left\lceil\log\left(\frac{2(r_{0,\max} +BR_{\max})}{(1-\gamma)\delta\tau}\right)/\left(\log\left(\tfrac{1}{\gamma}\right)\right)\right\rceil$.

\section{Details of Numerical Simulations}\label{sec:appendixnumericalsim}
\subsection{Reward and Constraint Costs Setting}
The interaction of the agent and virtual environment is divided into episodes and the horizon of each episode is $500$. The agent receives a reward $r_0=+1$ (reward for the objective), when the end of the lower link is at a height of $0.5$. It is penalized with cost $r_1=+1$ (cost for constraint 1)  when the first link swings at a anticlockwise direction and the agent applies a $+1$ torque to the actuating joint; it also receives a penalty of $r_2=+1$ (cost for constraint 2) when the second link swings at a anticlockwise direction with respect to the first link and the agent applies a $+1$ torque to the actuating joint. The cost thresholds for both $r_1$ and $r_2$ are set to be $50$.
\subsection{Parameters Setting}
The details of PDO and TRPO are given in \citep{schulman2015trust,achiam2017constrained}. For the PDO algorithm, we tune the stepsize (for dual variable update) from the set $\{0.0001, 0.0005, 0.001, 0.005, 0.01, 0.05\}$, and the best performance (given in \Cref{fig:fig}) corresponds to the choice of 0.0005. For our AR-CPO method, we tune the following parameters.
\begin{itemize}[leftmargin=*]
	\item $q$:  parameter in the update of $\underline\lambda$,
	\[\underline{\lambda}_t = (1- q_t)\bar{\lambda}_{t-1} + q_t \lambda_{t-1},\]
	where $t$ is the number of iterations.
	\item $\alpha$: parameter in the update of $\bar\lambda$,
	\[\bar{\lambda}_t = (1- \alpha_t)\bar{\lambda}_{t-1}  + \alpha_t \lambda_t.\]
	
	In our experiment, we set $\alpha_t$ and $q_t$ to first diminish and then remain constant as follows
\begin{equation*}
	\alpha_t = q_t =\begin{cases} \frac{2s}{t+1}, \qquad & 1\le t<H\\
	\frac{2s}{H}, &\text{otherwise}, \end{cases}
\end{equation*}
where the parameters $H>0$ and $0<s<1$ completely determine $\alpha_t$ and $q_t$.
	
	\item inner\_loop: the number of policy optimization steps in RegPO subroutine.
	\item max\_KL: the parameter that controls the RegPO updates.
	\item $\eta$: the parameter in the update of $\lambda$, defined in Step 7 of \Cref{alg:arcpo} with $\mu =0$ and $\Lambda = [0, \infty)^2$.
\end{itemize}

We tune the above parameters over a range of choices as $s \in \{1, 0.3, 0.05\}$, $H \in \{15, 45, 85\}$, inner\_loop $\in \{1, 5\}$, max\_KL $\in \{0.001, 0.003, 0.01\}$ and $\eta \in \{0.0001, 0.0003, 0.001\}$, and  
do the grid search over all possible choices. The best performance (given in \Cref{fig:fig}) corresponds to the choices of $s = 1$, $H= 45$, inner\_loop=$1$, max\_KL$=0.01$ and $\eta =0.0003$.

\section{Proofs of \Cref{thm1} and \Cref{corollary:totalcomplexity}}\label{proof:thm1}

\subsection{Main Proof of \Cref{thm1}}  
\noindent

\textit{\textbf{Step 1.}}
The first step establishes that the dual function $d_{\tau, \mu}(\lambda)$ is an $L_d$-smooth function in \Cref{prop:smoothness}, which is a necessary property for the complexity order improvement. The entropy regularizer plays an important role for such a property to hold, without which \Cref{prop:smoothness} does not hold anymore. 
\begin{prop}\label{prop:smoothness}
    Suppose \Cref{assumption:ergodicity} holds, then  $d_{\tau, \mu}(\lambda)$ is $L_d$-smooth:
    \begin{align}
        \|\nabla d_{\tau, \mu}(\lambda) - \nabla d_{\tau, \mu}(\lambda^\prime) \|_2 \le L_d \|\lambda - \lambda^\prime\|_2, \quad\forall\lambda, \lambda^\prime\in\mR_+^m,
    \end{align}
    where $L_d= \frac{2R_{\max}^2L_\nu}{(1-\gamma)^2\tau} +\mu$.
\end{prop}


        \textit{\textbf{Step 2.}} 
        This step contains three major novel components: (a) construction of a weighted value function $(1-\alpha)^{T-1} V_0^{\pi_1}(\rho)+ \sum\nolimits_{t=2}^T \alpha(1-\alpha)^{T-t} V_0^{\pi_{t}}(\rho)$, which will ultimately yield the convergence of the optimality gap in Step 4; (b) construction of a weighted policy $\tilde \pi$ based on the updates of the dual variable $\{\underline \lambda_t\}$ in Algorithm \ref{alg:arcpo}, whose value function equals to the weighted value function constructed in (a) as shown in \Cref{prop:tildepi}; and (c) establishment of the connection between the optimality gap, the constraint violation and the updates of the dual variable $\{\lambda_t\}$ as presented in \Cref{prop:connection}. 

\begin{prop}\label{prop:tildepi}
Define the policy $\tilde{\pi}$, such that, for all $s\in\mcs$ and $a\in\mca$,
\begin{align}\label{eq:tildepi}
    \tilde{\pi} (a|s) = \frac{(1-\alpha)^{T-1} \nu_\rho^{\pi_1}(s,a) + \sum\nolimits_{t=2}^T \alpha(1-\alpha)^{T-t} \nu_\rho^{\pi_{t}}(s,a)}{(1-\alpha)^{T-1} \chi^{\pi_1}_\rho(s) + \sum\nolimits_{t=2}^T \alpha(1-\alpha)^{T-t} \chi_\rho^{\pi_{t}}(s)},
\end{align}
where $\chi_\rho^\pi (s) = \sum_{a^\prime\in\mca} \nu_\rho^\pi(s, a^\prime)$. Then for $i=0,...,m,$
\begin{align}
    V_i^{\tilde\pi}(\rho) =(1-\alpha)^{T-1} V_{i}^{\pi_1}(\rho)  + \sum_{t=2}^T \alpha(1-\alpha)^{T-t} V_{i}^{\pi_{t}}(\rho).
\end{align}
\end{prop}
\begin{prop}\label{prop:connection}
Consider \Cref{alg:arcpo}. Then the summation of the optimality gap and the constraint violation is bounded by the dual variable increments as follows. For any $\lambda \in \Lambda$,
\begin{align}\label{eq:8}
		&V_0^*(\rho)  - \left[(1-\alpha)^{T-1} V_0^{\pi_1}(\rho)  + \sum\nolimits_{t=2}^T \alpha(1-\alpha)^{T-t} (V_0^{\pi_{t}}(\rho) + \lambda^{\top} (c- {V}^{\pi_{t}}(\rho)))\right]  \nonumber\\
		\leq&  \left(\sum\nolimits_{t=1}^{T} \alpha(1-\alpha)^{T-t}\|\lambda_t-\lambda_{t-1}\|_2^2\right)^{\frac{1}{2}}\left(\tfrac{q}{\alpha}\tfrac{R_{\max}}{1-\gamma} + \tfrac{R_{\max}}{1-\gamma} + \tfrac{2\sqrt{m}B}{\eta}\right)\nonumber\\
		&+ \sum\nolimits_{t=1}^{T} (1-\alpha)^{T-t}\Delta_{t} +\tfrac{\tau}{1-\gamma} \log |\mca| +6\mu m B^2.
		\end{align}
\end{prop}

		 The analysis indicates that the convergence rate of the distance between two consecutive iterates $\|\lambda_t-\lambda_{t-1}\|_2^2$ plays a central role in the overall convergence and complexity of the optimality gap and constraint violation of policy $\tilde \pi$, which we will bound in Step 3.
		 
		\textit{\textbf{Step 3.}} This step provides the convergence bound for the dual variable $\lambda$, which is updated via Nesterov's accelerated gradient descent. However, our analysis here is significantly different from the Nesterov's analysis. The reason is that the Nesterov's analysis would provide only the convergence in terms of the dual function gap for our problem, whereas for the {\bf constrained} MDP problem, our goal is the convergence guarantee in terms of the optimality gap and the constraint violation (see Definition \ref{def:gap}), which requires to bound the distance between two consecutive iterates of the dual variable due to \Cref{prop:connection}. To this end, in this step we develop the following proposition. 
\begin{prop}\label{prop:keyconvergenceprop}
	Let $(\underline{\lambda}_t,\lambda_t,\bar \lambda_t) \in \Lambda$ be generated by Algorithm \ref{alg:arcpo} and the parameters be chosen in \eqref{stepsize}.
	Then for any $\lambda \in \Lambda$,
	\begin{align}
	\sum_{t=1}^T \tfrac{\alpha (1-\alpha)^{T-t}}{4\eta}\|\lambda_t-\lambda_{t-1}\|_2^2 \leq (1-\alpha)^T K_0(\lambda) - K_T (\lambda)+ \sum_{t=1}^T \alpha (1-\alpha)^{T-t} \langle \delta_t, \lambda - \lambda_t \rangle,\label{convergence1}
	\end{align}
	where $K_t(\lambda) \coloneqq d_{\tau, \mu}(\bar \lambda_t)-d_{\tau, \mu}(\lambda)+\tfrac{\alpha}{2}(\mu+\tfrac{1}{\eta})\|\lambda_t-\lambda\|_2^2$, and $\delta_t \coloneqq\widehat \nabla d_{\tau,\mu}(\underline \lambda_t)-\nabla d_{\tau,\mu}(\underline \lambda_t)$ .
\end{prop}
\Cref{prop:keyconvergenceprop} indicates that by introducing the $\ell_2$ regularization for the dual variable, the distance term $\sum\nolimits_{t=1}^{T} \alpha(1-\alpha)^{T-t}\|\lambda_t-\lambda_{t-1}\|_2^2$ converges linearly in terms of the error $K_0(\lambda^*_{\tau,\mu})$ introduced by the selection of the initial point $\lambda_0$. The gradient error $\delta_t$ can be made as small as possible via sufficient RegPO updates.

        
        \textit{\textbf{Step 4.}} This step characterizes the ultimate convergence in terms of the optimality gap and the constraint violation. 
        Substituting the bound in \cref{convergence1} into \cref{eq:8} yields that
        \begin{align}
		&V_0^*(\rho)  - \left[(1-\alpha)^{T-1} V_0^{\pi_1}(\rho)  + \sum\nolimits_{t=2}^T \alpha(1-\alpha)^{T-t} (V_0^{\pi_{t}}(\rho) + \lambda^{\top} ( c- {V}^{\pi_{t}}(\rho)))\right]  \nonumber\\
		\leq& \left(4\eta(1-\alpha)^T K_0(\lambda^*_{\tau,\mu}) + \sum\nolimits_{t=1}^T 4\eta\alpha (1-\alpha)^{T-t} \langle \delta_t, \lambda^*_{\tau,\mu} - \lambda_t \rangle\right)^{\frac{1}{2}}\left(\tfrac{q}{\alpha}\tfrac{R_{\max}}{1-\gamma} + \tfrac{R_{\max}}{1-\gamma} + \tfrac{2\sqrt{m}B}{\eta}\right)\nonumber\\
		&+ \sum\nolimits_{t=1}^{T} (1-\alpha)^{T-t}\Delta_{t} +\tfrac{\tau}{1-\gamma} \log |\mca| +6\mu m B^2\nonumber\\
		\leq& \left(2(1-\alpha)^{\frac{T}{2}} \sqrt{\eta K_0(\lambda^*_{\tau,\mu})} + \sqrt{8\eta\sqrt{m}B \max_{t=1,\ldots T} \|\delta_t\|_2}\right) \left(  \tfrac{2R_{\max}}{1-\gamma} + \tfrac{2\sqrt{m}B}{\eta}\right)\nonumber\\
		&+ \tfrac{\max_{t=1,\ldots T}  |\Delta_{t}|}{\alpha} +\tfrac{\tau}{1-\gamma} \log |\mca| +6\mu m B^2.\label{d_6}
		\end{align}
        
        The result of optimality gap follows directly by taking $\lambda=0$ in \cref{d_6}. 
		
		To further derive the bound on the constraint violation, we first add $(1-\alpha)^{T-1}\lambda^\top({c}-{V}^{\pi_1}(\rho) )$ on both sides of \cref{d_6}. By taking $\lambda_{i}= 2B$ if $({c}-{V}^{\tilde \pi}(\rho))_{i}> 0$, and taking $\lambda_{i}  = 0$ otherwise, we have
		\begin{align}
		V_0^*(\rho)-&V^{\tilde \pi}_0(\rho) + 2B \mathbf{1}^\top ({c}-{V}^{\tilde \pi}(\rho))_+\nonumber\\
		\leq&\left(2(1-\alpha)^{\frac{T}{2}} \sqrt{\eta K_0(\lambda^*_{\tau,\mu})} + \sqrt{8\eta\sqrt{m}B \max_{t=1,\ldots T} \|\delta_t\|_2}\right) \left( \tfrac{2R_{\max}}{1-\gamma} + \tfrac{2\sqrt{m}B}{\eta}\right)\nonumber\\
		&+ \tfrac{\max_{t=1,\ldots T}  |\Delta_{t}|}{\alpha} +\tfrac{\tau}{1-\gamma}\log |\mca| +6\mu m B^2 + 2(1-\alpha)^{T-1}\sqrt{m}B\tfrac{R_{\max}}{1-\gamma}.\label{eq:final}
		\end{align}
		
		Then we utilize the following proposition to bound the constraint violation.
		\begin{prop} \label{prop:constraintviolation}
    Given a policy $\tilde\pi$, suppose 
    $V^{*}_0(\rho) - V^{\tilde\pi}_0(\rho) + \langle\lambda, ({c} - V^{\tilde\pi})_+\rangle \le \sigma,$
    where $\lambda\in\mR^m_+$ satisfying $\lambda\succeq\lambda^* +B \mathbf{1}$. Then, we have $\|({c} - V^{\tilde\pi})_+\|_1\le \frac{\sigma}{B}$.
\end{prop}
Together with \Cref{prop:constraintviolation}, \cref{eq:final}  implies the bound on the constraint violation.
		\begin{align*}
		\|({c}-{V}^{\tilde \pi}(\rho))_+\|_1 \leq& \left(2\left(1-\sqrt{\tfrac{\mu}{2L_d}}\right)^{\frac{T}{2}} \sqrt{\eta K_0(\lambda^*_{\tau,\mu})} + \sqrt{8\eta\sqrt{m}B \max_{t=1,\ldots T} \|\delta_t\|_2}\right) \left(\tfrac{2R_{\max}}{(1-\gamma)B} + \tfrac{2\sqrt{m}}{\eta}\right)\\
		&+ \sqrt{\tfrac{2L_d}{\mu}}\tfrac{\max_{t=1,\ldots T}  |\Delta_{t}|}{B} +\tfrac{\tau \log |\mca|}{(1-\gamma)B} +6\mu m B + 2\left(1-\sqrt{\tfrac{\mu}{2L_d}}\right)^{T-1}\tfrac{\sqrt{m}R_{\max}}{1-\gamma}.
		\end{align*}
		The complete proof of Propositions \ref{prop:smoothness}, \ref{prop:tildepi}, \ref{prop:connection}, \ref{prop:keyconvergenceprop} and \ref{prop:constraintviolation} are provided in \Cref{section:proofofproposition}.

\subsection{Supporting Lemmas}\label{section:supportinglemmas}
We first introduce three useful lemmas. They are quite explicit and thus we omit the proof. 

\begin{newlemma}\label{lemma:logfunction}
Given $x,y \in (0, 1)$, we have $|x- y| \le |\log(x) - \log(y)|$.
\end{newlemma}
\begin{newlemma}\label{lemma:expfunction}
    For any two vectors $x,y\in\mR^m$, $\left|\log\left(\|\exp(x)\|_1\right) - \log\left(\|\exp(y)\|_1\right)\right| \le \|x -y\|_\infty.$ 
\end{newlemma}

\begin{newlemma}\label{lemma:maximumfunction}
Suppose we have two functions $f$ and $g$: $\mcx\rightarrow\mR$, where $\mcx$ could be a finite set or a subset of $\mR^m$ for some $m\in\mathbb{N}$. Then, we have 
\begin{equation*}
    \left|\sup_{x\in \mcx} \{f(x)\} - \sup_{x\in \mcx} \{g(x)\}\right| \le \sup_{x\in\mcx} \left\{|f(x) - g(x)|\right\}.
\end{equation*}
\end{newlemma}

The following lemma \citep[Lemma 3.5]{lan2020first}, often referred to as the ``three-point lemma", characterizes the property of one step update in Line 7 of \Cref{alg:arcpo}.
\begin{newlemma}\label{lemma:threepointlemma}
	Let $(\underline{\lambda}_t,\lambda_t,\bar \lambda_t) \in \Lambda$ be generated by Algorithm \ref{alg:arcpo}, then for any $\lambda\in \Lambda$, we have
	\begin{align}
	\eta \langle \widehat \nabla d_{\tau,\mu}(\underline{\lambda}_t), \lambda_t -\lambda \rangle + \tfrac{\eta \mu}{2}\|\lambda_t-\underline \lambda_{t}\|_2^2  + \tfrac{1}{2} \|\lambda_t-\lambda_{t-1}\|_2^2 \leq \tfrac{\eta \mu}{2}\|\lambda-\underline \lambda_{t}\|_2^2  + \tfrac{1}{2} \|\lambda-\lambda_{t-1}\|_2^2.\label{3p}
	\end{align}
\end{newlemma}

\begin{newlemma}\label{lemma:1}
Suppose \Cref{assumption:slatercondition} holds. Let $\lambda^*$ be the minimizer of the dual function. Then \[\|\lambda^*\|_1\le B\coloneqq\tfrac{r_{0,\max}}{(1-\gamma)\xi}.\]
\end{newlemma}
\begin{proof} By the definition of $d(\lambda)$, we have that
\begin{align}\label{eq:dualfunctionbound}
    d(\lambda)\ge V_0^{\pi_\xi}(\rho) + \langle \lambda, V^{\pi_\xi}(\rho) - \cb \rangle \overset{(i)}\ge \xi\|\lambda\|_1,
\end{align}
where $(i)$ follows from the fact that $V_0^{\pi_\xi}\ge 0$ and \Cref{assumption:slatercondition}.
Moreover, constrained reinforcement learning has no duality gap \citep{paternain2019constrained}. Combining this with the fact that the optimal function value is bounded, we conclude that the optimal dual variable exists in $\mR_+^m$. Taking $\lambda=\lambda^*$ in \cref{eq:dualfunctionbound} and using the fact that $d(\lambda^*) = V_{0}^{\pi^*}(\rho)\le \tfrac{r_{0,\max}}{1-\gamma}$, we complete the proof.
\end{proof}

\begin{newlemma}\label{lemma_nu_lipschitz}
Under \Cref{assumption:ergodicity}, for any $ \lambda, \lambda' \in \mathbb{R}_+^m$, and the corresponding optimal policy $\pi^*_{\tau,\lambda}$ and $\pi^*_{\tau,\lambda'}$, we have
\begin{equation*}
\left\|\nu^{\pi^*_{\tau,\lambda}}_\rho - \nu^{\pi^*_{\tau,\lambda'}}_\rho\right\|_1 = 2\dtv\left(\nu^{\pi^*_{\tau,\lambda}}_\rho, \nu^{\pi^*_{\tau,\lambda'}}_\rho\right)\le L_{\nu}\max_{s\in \mcs}\left\|\pi^*_{\tau,\lambda}(\cdot|s)-\pi^*_{\tau,\lambda'}(\cdot|s)\right\|_1,
\end{equation*}
where $L_{\nu}\coloneqq \lceil \log_{\beta}(C_M^{-1})\rceil + (1-\beta)^{-1} + 1$.
\end{newlemma}
\begin{proof}
Our proof here is inspired by \cite{zou2019finite,xu2020improving}. Readers may find more bounds on the difference of visitation measures in their papers. Denote $\tilde P(s'|s,a) := \gamma P(s'|s,a) + (1-\gamma) \rho(s')$, $\forall s, s' \in \mcs,~ a \in \mca.$ Denote $\tilde P^\pi$ as a transition kernel, where $\tilde P^\pi_{s,s'}:= \gamma \tsum_{a\in \mca} \pi(a|s) \mathsf{P}(s'|s,a)+(1-\gamma) \rho(s')$. 
Moreover, recall that $\chi^\pi_\rho(s) = \sum_{a\in\mca}\nu_\rho^\pi(s, a) = (1-\gamma)\tsum_{t=0}^\infty \gamma^t \mathsf{P}(s_t=s|s_0 \sim \rho,a_t \sim \pi(\cdot|s_t), s_{t+1} \sim \mathsf{P}(\cdot|s_t,a_t))$. Then we have
\begin{align*}
    \left(\left(\chi_\rho^\pi \right)^\top\tilde P^\pi\right)_{s'} &= \gamma \tsum_{s\in \mcs}(1-\gamma)\big(\tsum_{t=0}^\infty \gamma^t \mathsf{P}(s_t=s|s_0 \sim \rho,a_t \sim \pi(\cdot|s_t), s_{t+1} \sim \mathsf{P}(\cdot|s_t,a_t)) \big)\cdot \big(\tsum_{a\in \mca} \pi(a|s) \mathsf{P}(s'|s,a) \big)\\
    &\quad + (1-\gamma) \rho(s')\big(\tsum_{s\in S} \chi_\rho^\pi (s)\big)\\
    &= (1-\gamma) \tsum_{t=1}^\infty \gamma^t \mathsf{P}(s_t=s'|s_0 \sim \rho,a_t \sim \pi(\cdot|s_t), s_{t+1} \sim \mathsf{P}(\cdot|s_t,a_t))  + (1-\gamma) \rho(s')\big(\tsum_{s\in S} \chi_\rho^\pi (s)\big)\\
    &= (1-\gamma)\left(\rho(s^\prime) + \tsum_{t=1}^\infty \gamma^t \mathsf{P}(s_t=s'|s_0 \sim \rho,a_t \sim \pi(\cdot|s_t), s_{t+1} \sim \mathsf{P}(\cdot|s_t,a_t)) \right)\\
    &=\chi_\rho^\pi(s').
\end{align*}
The above equality implies that $\left(\chi_\rho^\pi \right)^\top\tilde P^\pi= \left(\chi_\rho^\pi \right)^\top$. Therefore, $\chi_\rho^\pi$ is the stationary distribution with respect to kernel $\tilde P^\pi$. Invoking Theorem 3.1 of \cite{mitrophanov2005sensitivity} gives us 
\begin{align}
    \left\|\chi_\rho^{\pi^*_{\tau,\lambda}}- \chi_\rho^{\pi^*_{\tau,\lambda'}}\right\|_1\leq \left(\lceil \log_{\beta}(C_M^{-1})\rceil + (1-\beta)^{-1}\right)\|\tilde P^{\pi^*_{\tau,\lambda}} - \tilde P^{\pi^*_{\tau,\lambda'}}\|,\label{eq:midlipschitz0}
\end{align}
where $\|\tilde P^{\pi^*_{\tau,\lambda}} - \tilde P^{\pi^*_{\tau,\lambda'}}\|:=\sup_{\|q\|_1=1} \tsum_{s'\in \mcs}\big|\tsum_{s\in \mcs}q(s)(\tilde P^{\pi^*_{\tau,\lambda}}_{s,s'}-\tilde P^{\pi^*_{\tau,\lambda'}}_{s,s'})\big|$.\\
We  upper-bound $\|\tilde P^{\pi^*_{\tau,\lambda}} - \tilde P^{\pi^*_{\tau,\lambda'}}\|$ in the next steps.
\begin{align}
\|\tilde P^{\pi^*_{\tau,\lambda}} - \tilde P^{\pi^*_{\tau,\lambda'}}\|&\leq \sup_{\|q\|_1=1} \tsum_{s'\in \mcs}\tsum_{s\in \mcs}|q(s)|\big|(\tilde P^{\pi^*_{\tau,\lambda}}_{s,s'}-\tilde P^{\pi^*_{\tau,\lambda'}}_{s,s'})\big|\nonumber\\
&\leq \gamma\sup_{\|q\|_1=1} \tsum_{s\in \mcs}|q(s)| \tsum_{a\in \mca}\big|\pi^*_{\tau,\lambda}(a|s)-\pi^*_{\tau,\lambda'}(a|s)\big|\big(\tsum_{s'\in \mcs} \mathsf{P}(s'|s,a)\big)\nonumber\\
&\le \sup_{\|q\|_1=1} \big(\tsum_{s \in S}|q(s)|\big) \cdot \max_{s\in \mcs} \|\pi^*_{\tau,\lambda}(\cdot|s)-\pi^*_{\tau,\lambda'}(\cdot|s)\|_1\nonumber\\
&= \max_{s\in \mcs} \|\pi^*_{\tau,\lambda}(\cdot|s)-\pi^*_{\tau,\lambda'}(\cdot|s)\|_1.\nonumber
\end{align}

Substituting the above inequality in to \cref{eq:midlipschitz0} yields
\begin{align}
 \left\|\chi_\rho^{\pi^*_{\tau,\lambda}}- \chi_\rho^{\pi^*_{\tau,\lambda'}}\right\|_1\leq \left(\lceil \log_{\beta}(C_M^{-1})\rceil + (1-\beta)^{-1}\right)\max_{s\in \mcs} \|\pi^*_{\tau,\lambda}(\cdot|s)-\pi^*_{\tau,\lambda'}(\cdot|s)\|_1.\label{eq:midlipschitz1}
\end{align}
Moreover, we have
\begin{align*}
    \left\|\nu_\rho^{\pi^*_{\tau,\lambda}}- \nu_\rho^{\pi^*_{\tau,\lambda'}}\right\|_1 &= \sum_{s\in\mcs, a\in\mcs} \left|\nu_\rho^{\pi^*_{\tau,\lambda}}(s,a)- \nu_\rho^{\pi^*_{\tau,\lambda'}}(s,a)\right|\\
    &= \sum_{s\in\mcs, a\in\mcs} \left|\chi_\rho^{\pi^*_{\tau,\lambda}}(s){\pi^*_{\tau,\lambda}}(a|s)- \chi_\rho^{\pi^*_{\tau,\lambda'}}(s){\pi^*_{\tau,\lambda'}}(a|s)\right|\\
    & =  \sum_{s\in\mcs, a\in\mcs} \left|\chi_\rho^{\pi^*_{\tau,\lambda}}(s){\pi^*_{\tau,\lambda}}(a|s)-   \chi_\rho^{\pi^*_{\tau,\lambda'}}(s){\pi^*_{\tau,\lambda}}(a|s) + \chi_\rho^{\pi^*_{\tau,\lambda'}}(s){\pi^*_{\tau,\lambda}}(a|s) - \chi_\rho^{\pi^*_{\tau,\lambda'}}(s) {\pi^*_{\tau,\lambda'}}(a|s)\right|\\
    &\le \sum_{s\in\mcs, a\in\mcs} \left|\chi_\rho^{\pi^*_{\tau,\lambda}}(s)- \chi_\rho^{\pi^*_{\tau,\lambda'}}(s)\right|{\pi^*_{\tau,\lambda}}(a|s) + \sum_{s\in\mcs, a\in\mcs} \chi_\rho^{\pi^*_{\tau,\lambda'}}(s)\left|{\pi^*_{\tau,\lambda}}(a|s)  -  {\pi^*_{\tau,\lambda'}}(a|s)\right|\\
    & = \sum_{s\in\mcs} \left|\chi_\rho^{\pi^*_{\tau,\lambda}}(s)- \chi_\rho^{\pi^*_{\tau,\lambda'}}(s)\right| + \sum_{s\in\mcs} \chi_\rho^{\pi^*_{\tau,\lambda'}}(s)\left\| {\pi^*_{\tau,\lambda}}(\cdot|s)  -  {\pi^*_{\tau,\lambda'}}(\cdot|s)\right\|_1\\
    &\overset{(i)}\le \left(\lceil \log_{\beta}(C_M^{-1})\rceil + (1-\beta)^{-1} + 1\right) \max_{s\in \mcs} \|\pi^*_{\tau,\lambda}(\cdot|s)-\pi^*_{\tau,\lambda'}(\cdot|s)\|_1,
\end{align*}
where $(i)$ follows from \cref{eq:midlipschitz1}.
\end{proof}

\begin{newlemma}\label{lemma:policycontinuous}
The optimal policy $\pi^*_{\tau, \lambda}$ for the regularized MDP is smooth with respect to $\lambda$, i.e., for all $\lambda, \lambda^\prime \in\mR^m_+$, we have 
\begin{equation*}
      \max_{s\in \mcs} \|\pi^*_{\tau,\lambda}(\cdot|s)-\pi^*_{\tau,\lambda'}(\cdot|s)\|_1\le  \tfrac{2R_{\max}}{(1-\gamma)\tau}\|\lambda -\lambda^\prime\|_2.
\end{equation*}
\end{newlemma}
\begin{proof} 
We first rewrite the Lagrangian as follows
\begin{align}\label{eq:aaa}
    \mcL_{\tau,\mu}(\pi, \lambda) &=  V_{\tau,\lambda}^{\pi} (\rho)- \langle\lambda, \cb\rangle +\tfrac{\mu}{2}\|\lambda\|^2_2.
\end{align}
The optimal policy that maximizes eq. \eqref{eq:aaa} is the optimal policy for the same MDP but with entropy regularizer and a different reward function $r_\lambda$, i.e. $\pi^*_{\tau, \lambda}$. In \cite{nachum2017bridging} it was shown that $\pi^*_{\tau, \lambda}$ is as follows
\begin{align*}
    \pi_{\tau,\lambda}^* (a|s) = \tfrac{\exp(Q_{\tau,\lambda}^*(s,a)/\tau)}{\sum_{a^\prime \in\mca}\exp(Q_{\tau,\lambda}^*(s,a^\prime)/\tau)}, \quad \forall s\in\mcs, a\in\mca .
\end{align*}
Thus, for any $s\in\mcs$ and $a\in\mca$, we have that
\begin{align*}
    &|\log\left(\pi^*_{\tau, \lambda}(s,a)\right) - \log\left(\pi^*_{\tau, \lambda^\prime}(s,a)\right)| \nn\\
    &=  \left|\tfrac{1}{\tau}(Q_{\tau, \lambda}^*(s,a) -Q_{\tau, \lambda^\prime}^*(s,a)) - \left(\log\left(\sum_{a^\prime\in\mca} \exp\left(\tfrac{Q_{\tau, \lambda}^*(s,a^\prime)}{\tau}\right)\right) - \log\left(\sum_{a^\prime\in\mca} \exp\left(\tfrac{Q_{\tau, \lambda^\prime}^*(s,a^\prime)}{\tau}\right)\right)\right)\right|\\
    &\overset{(i)}\le \tfrac{1}{\tau}|Q_{\tau, \lambda}^*(s,a) -Q_{\tau, \lambda^\prime}^*(s,a)| +  \tfrac{1}{\tau} \|Q_{\tau, \lambda}^*(s,\cdot) -Q_{\tau, \lambda^\prime}^*(s,\cdot)\|_\infty\\
    &\le \tfrac{2}{\tau}\|Q_{\tau, \lambda}^* -Q_{\tau, \lambda^\prime}^*\|_\infty,
\end{align*}
where $(i)$ follows from the fact that $|x+y|\le |x| + |y|$ and \Cref{lemma:expfunction}. Further applying \Cref{lemma:logfunction}, we have that
\begin{align}
    |\pi^*_{\tau, \lambda}(s,a) - \pi^*_{\tau, \lambda^\prime}(s,a)| &\le\tfrac{2}{\tau}\|Q_{\tau, \lambda}^* -Q_{\tau, \lambda^\prime}^*\|_\infty\nonumber\\
    &= \tfrac{2}{\tau} \max_{s\in\mcs, a\in\mca} |Q_{\tau, \lambda}^*(s,a) -Q_{\tau, \lambda^\prime}^* (s,a)|\nonumber\\
    &\overset{(i)}\le \tfrac{2}{\tau} \max_{s\in\mcs, a\in\mca} \max_{\pi\in\Pi}|Q_{\tau, \lambda}^\pi(s,a) -Q_{\tau, \lambda^\prime}^\pi(s,a)|\nonumber\\
    &\overset{(ii)}\le \tfrac{2R_{\max}}{(1-\gamma)\tau} \|\lambda -\lambda^\prime\|_2, \label{eq:pidifference}
\end{align}
where $(i)$ follows from \Cref{lemma:maximumfunction}, and $(ii)$ follows from the uniform upper bound developed below,
\begin{align*}
    &|Q_{\tau, \lambda}^\pi(s,a) -Q_{\tau, \lambda^\prime}^\pi(s,a)|\\
    &\quad= \left|r_{\lambda}(s,a) + \gamma\mE_{s^\prime\sim\mathsf{P}(\cdot|s,a)}\left[V^\pi_{\tau,\lambda}(s^\prime)\right] - r_{\lambda^\prime}(s,a) - \gamma\mE_{s^\prime\sim\mathsf{P}(\cdot|s,a)}\left[V^\pi_{\tau,\lambda^\prime}(s^\prime)\right] \right|\\
    &\quad\le |r_{\lambda}(s,a) - r_{\lambda^\prime}(s,a)| + \gamma\mE_{s^\prime\sim\mathsf{P}(\cdot|s,a)}\left[\left|V^\pi_{\tau,\lambda}(s^\prime) - V^\pi_{\tau,\lambda^\prime}(s^\prime)\right|\right]\\
    &\quad\le \tfrac{R_{\max}}{1-\gamma}\|\lambda - \lambda^\prime\|_2.
\end{align*}
Taking maximum on the left-hand-side of \cref{eq:pidifference} concludes the proof.
\end{proof}

\subsection{Proof of the Propositions}\label{section:proofofproposition}

\subsubsection{Proof of \Cref{prop:smoothness}} 
 
We first note that $\Pi$ is a compact set. Moreover, $\mcL_{\tau,\mu}(\pi, \lambda)$ is continuous in both $\pi$ and $\lambda$, and is convex in $\lambda$ for any given $\pi$. The maximum of $\mcL_{\tau,\mu}(\pi, \lambda)$ is attained uniquely at $\pi^*_{\tau,\lambda}$. Thus, by Danskin's theorem, we obtain that
 \[\nabla d_{\tau, \mu}(\lambda) = \tfrac{\partial}{\partial \lambda} \mcL_{\tau, \mu} (\pi, \lambda)|_{\pi = \pi^*_{\tau,\lambda}}= V^{\pi^*_{\tau,\lambda}}(\rho) -\cb +\mu\lambda.\]
 Moreover, it can be shown that
 \begin{align*}
     \|\nabla d_{\tau, \mu}(\lambda) - \nabla d_{\tau, \mu}(\lambda^\prime) \|_2 &= \|V^{\pi^*_{\tau,\lambda}}(\rho) -\cb +\mu\lambda - (V^{\pi^*_{\tau,\lambda^\prime}}(\rho) -\cb +\mu\lambda^\prime)\|_2\nonumber\\
     &\le \|V^{\pi^*_{\tau,\lambda}}(\rho)- V^{\pi^*_{\tau,\lambda^\prime}}(\rho)\|_2 +\mu \|\lambda - \lambda^\prime\|_2\nonumber\\
     & = \tfrac{1}{1-\gamma}\bigg(\sum_{i=1}^m (\langle\nu_\rho^{\pi^*_{\tau,\lambda}}- \nu_\rho^{\pi^*_{\tau,\lambda^\prime}}, r_i\rangle)^2\bigg)^{1/2} + \mu \|\lambda - \lambda^\prime\|_2\nn\\
     &\leq \tfrac{R_{\max}}{1-\gamma}\|\nu_{\rho}^{\pi^*_{\tau,\lambda}}- \nu_{\rho}^{\pi^*_{\tau,\lambda'}}\|_1 + \mu\|\lambda-\lambda'\|_2\nn\\
     &\overset{(i)}\leq\tfrac{R_{\max}L_\nu}{1-\gamma}\max_{s\in \mcs}\|\pi^*_{\tau,\lambda}(\cdot|s)-\pi^*_{\tau,\lambda'}(\cdot|s)\|_1+ \mu\|\lambda-\lambda'\|_2\nn\\
     &\overset{(ii)}\leq\left(\tfrac{2R^2_{\max}L_\nu}{(1-\gamma)^2\tau}+\mu\right)\|\lambda-\lambda'\|_2,
 \end{align*}
 where $(i)$ follows from \Cref{lemma_nu_lipschitz} and $(ii)$ follows from \Cref{lemma:policycontinuous}. Then setting $L_d =\tfrac{2R_{\max}^2L_\nu}{(1-\gamma)^2\tau} +\mu$ completes the proof.

\subsubsection{Proof of \Cref{prop:tildepi}}
First, we construct an auxiliary non-stationary policy $\pi_N$ as follows. Let $N$ be a random index variable with $\prob(N=1) = (1-\alpha)^{T-1}$, and $\prob(N= i) = \alpha(1-\alpha)^{T-i}$ for $i= 2,\dots, T$. To execute the algorithm, an agent first randomly generates a value for $N$, and then applies the policy $\pi_N(\cdot|s_t)$ throughout the process. By construction, we have
\begin{align}
&\nu_{\rho}^{\pi_N}(s,a)= (1-\alpha)^{T-1} \nu_\rho^{\pi_1}(s,a) + \sum\nolimits_{t=2}^T \alpha(1-\alpha)^{T-t} \nu_\rho^{\pi_{t}}(s,a), \\
&\chi_{\rho}^{\pi_N}(s)= (1-\alpha)^{T-1} \chi_\rho^{\pi_1}(s) + \sum\nolimits_{t=2}^T \alpha(1-\alpha)^{T-t} \chi_\rho^{\pi_{t}}(s).
\end{align}
Then by the definition of $\tilde{\pi}$, we have
\begin{align}\label{eq:pin}
\nu_{\rho}^{\pi_N}(s,a) = \chi_{\rho}^{\pi_N}(s) \tilde{\pi}(a|s).
\end{align}
Furthermore, applying $\tilde{\pi}$ as a stationary policy, we obtain
\begin{align}\label{eq:pitilde}
\nu_\rho^{\tilde\pi}(s, a)=\chi_\rho^{\tilde\pi} (s) \tilde\pi(a|s).   
\end{align}
Then it suffices to show that $\chi_{\rho}^{\pi_N}(s)=\chi_\rho^{\tilde\pi} (s)$, because \cref{eq:pin} and \cref{eq:pitilde} together imply 
\begin{equation*}
    \nu_\rho^{\tilde\pi}(s, a) = \nu_\rho^{\pi_N} (s, a) = (1-\alpha)^{T-1} \nu_\rho^{\pi_1}(s,a) + \sum\nolimits_{t=2}^T \alpha(1-\alpha)^{T-t} \nu_\rho^{\pi_{t}}(s,a), 
\end{equation*}
which further yields that $V_i^{\tilde\pi}(\rho) = (1-\alpha)^{T-1} V_i^{\pi_1}(\rho)  + \sum\nolimits_{t=2}^T \alpha(1-\alpha)^{T-t} V_i^{\pi_{t}}(\rho)$ holds for all $i=0,\ldots, m$ due to the fact that $V_i^{\pi}(\rho) = \left\langle r_i, \nu_\rho^\pi\right\rangle$.

To show that $\chi_{\rho}^{\pi_N}(s)=\chi_\rho^{\tilde\pi} (s)$, we will prove that they are both equal to the unique fixed point of a Bellman operator. The following proof is the one in \cite{csaba} and we include them here for completeness. To this end, for a policy $\pi$ and a  measure $\nu$ over the state space $\mcs$, we define a Bellman operator $\mathcal{B}_{\rho}^\pi$ as: for all $s\in\mcs$,
\begin{align*}
    (\mathcal{B}_{\rho}^\pi (\chi))_{s} \coloneqq (1-\gamma)\rho(s) + \gamma\sum_{s^\prime, a^\prime} \mathsf{P}(s|s^\prime, a^\prime)\pi(a^\prime|s^\prime)\chi(s^\prime).
\end{align*}
It is easy to verify that $\mathcal{B}_{\rho}^\pi$ satisfies the $\gamma$-contraction property with respect to $\ell_\infty$ norm, i.e., for any $\chi, \chi^\prime\in\mR^{|\mcs|}$, we have 
\begin{equation*}
    \|\mathcal{B}_{\rho}^\pi (\chi) - \mathcal{B}_{\rho}^\pi (\chi^\prime)\|_\infty \le \gamma \|\chi - \chi^\prime\|_\infty.
\end{equation*}
Thus, by the Banach–Caccioppoli fixed-point theorem, on the compact bounded set (the probability simplex over $\mcs$), $\mathcal{B}_{\rho}^\pi$ has a unique fixed point. We first show that $\chi_\rho^\pi$ is a fixed point of $\mathcal{B}_{\rho}^\pi$ for any $\pi$ as follows. 
\begin{align*}
    \chi_\rho^\pi(s) &= (1-\gamma)\sum_{t=0}^\infty \gamma^t \prob\{s_t = s| s_0 \sim \rho, \pi\}\\
    &= (1-\gamma)\rho(s) + (1-\gamma)\sum_{t=1}^\infty \gamma^t \prob(s_t = s| s_0 \sim \rho, \pi)\\
    &\overset{(i)}= (1-\gamma)\rho(s) + (1-\gamma)\sum_{t=1}^\infty \gamma^t\sum_{s^\prime, a^\prime} \big(\prob(s_t=s|s_{t-1}=s^\prime, a_{t-1}=a^\prime)\\
    &\qquad \quad \quad\qquad \quad \quad\cdot\prob(s_{t-1}= s^\prime, a_{t-1} = a^\prime|s_0\sim\rho, \pi)\big)\\
    &=  (1-\gamma)\rho(s) + \gamma\sum_{s^\prime, a^\prime}\prob(s_t=s|s_{t-1} = s^\prime, a_{t-1} = a^\prime)\pi(a^\prime|s^\prime)\\
    &\qquad \quad \quad\qquad \quad \quad \cdot(1-\gamma)\big(\sum_{t=1}^\infty \gamma^{t-1}\prob(s_{t-1}= s^\prime|s_0\sim\rho, \pi)\big)\\
    &=  (1-\gamma)\rho(s) + \gamma\sum_{s^\prime, a^\prime} \mathsf{P}(s|s^\prime, a^\prime){\pi}(a^\prime|s^\prime) \chi_\rho^\pi(s^\prime)\\
    &= (\mathcal{B}^\pi_\rho (\chi^\pi_\rho))_{s},
\end{align*}
where $(i)$ follows from the property of Markov chain. Thus, $\chi_\rho^{\tilde\pi}$ is a fixed point of $\mathcal{B}_{\rho}^{\tilde\pi}$.

 
Then, it is sufficient to show that $\chi_\rho^{\pi_N}$ is also a fixed point of $\mathcal{B}^{\tilde \pi}_\rho$, which is given as follows:
\begin{align*}
    \chi_{\rho}^{\pi_N}(s)&= (1-\gamma)\rho(s) + (1-\gamma)\sum_{t=1}^\infty \gamma^t \prob(s_t = s|s_0\sim\rho, \pi_N)\\
    &= (1-\gamma)\rho(s) + (1-\gamma)\sum_{t=1}^\infty \gamma^t \sum_{s^\prime, a^\prime}\bigg(\prob(s_{t-1} = s^\prime, a_{t-1}= a^\prime |s_0\sim\rho, \pi_N)\\
    &\qquad\qquad\qquad\qquad\qquad\qquad\qquad\cdot\prob(s_t=s |
    s_{t-1} = s^\prime, a_{t-1} = a^\prime)\bigg)\\
    &\overset{(i)} =  (1-\gamma)\rho(s) + \gamma\sum_{s^\prime, a^\prime} \mathsf{P}(s|s^\prime, a^\prime)\nu_\rho^{\pi_N}(s^\prime, a^\prime)\\
    &\overset{(ii)}= (1-\gamma)\rho(s) + \gamma\sum_{s^\prime, a^\prime} \mathsf{P}(s|s^\prime, a^\prime)\tilde{\pi}(a^\prime|s^\prime)\xi_\rho^{\pi_N}(s^\prime, a^\prime)\\
    &= (\mathcal{B}^{\tilde{\pi}}_\rho (\chi^{\pi_N}_\rho))_{s},
\end{align*}
where $(i)$ follows by rearranging the summation and $(ii)$ follows from the definitions of $\nu^{\pi_N}_\rho(s,a)$ and $\chi_\rho^{\pi_N}(s)$. Thus, $\chi_{\rho}^{\pi_N}(s)=\chi_\rho^{\tilde\pi} (s)$ because they both equal the unique fixed point of $\mathcal{B}^{\tilde \pi}_\rho$. This yields the desired result as we argue above, and concludes the proof.

\subsubsection{Proof of \Cref{prop:connection}} 
For convenience of the proof, define following notations
\begin{align*}
    \mcL_{\tau}(\pi, \lambda) = \mcL(\pi, \lambda) +\tau\mcH(\pi),
\end{align*}
and 
\begin{align*}
    d_{\tau}(\lambda) = \max_{\pi\in\Pi}\left\{\mcL_{\tau}(\pi, \lambda)\right\}.
\end{align*}

First, we build the recursive relationship for $\{d_\tau(\underline \lambda_t)\}$. By the definition of $d_{\tau}(\cdot)$, for $t=1,...,T-1$, we have
\begin{align*}
d_{\tau}(\underline \lambda_{t+1}) &= V_0^{\pi^*_{\tau,\underline \lambda_{t+1}}}(\rho)  + \underline \lambda_{t+1}^\top({V}^{\pi^*_{\tau,\underline \lambda_{t+1}}}(\rho)- c)\\
&=V_0^{\pi_{t+1}}(\rho) + \underline \lambda_{t+1}^\top ({V}^{\pi_{t+1}}(\rho)- c) + \Delta_{t+1}\\
&\overset{(i)}=(1-\alpha)[V_0^{\pi_{t+1}}(\rho) + \underline \lambda_{t}^\top ({V}^{\pi_{t+1}}(\rho)- c) ] \\
&\quad+ [(1-q)\alpha + q] [V_0^{\pi_{t+1}}(\rho) + \lambda_{t}^\top ({V}^{\pi_{t+1}}(\rho)- c) ] \\
&\quad-q(1-\alpha)[V_0^{\pi_{t+1}}(\rho) + \lambda_{t-1}^\top ({V}^{\pi_{t+1}}(\rho)- c) ] + \Delta_{t+1}\\
&\overset{(ii)}\leq (1-\alpha)d_{\tau}(\underline \lambda_{t}) + [(1-q)\alpha + q] [V_0^{\pi_{t+1}}(\rho) + \lambda_{t}^\top ({V}^{\pi_{t+1}}(\rho)- c) ]\\
&\quad -q(1-\alpha)[V_0^{\pi_{t+1}}(\rho) + \lambda_{t-1}^\top ({V}^{\pi_{t+1}}(\rho)- c) ] + \Delta_{t+1},
\end{align*}
where $(i)$ follows from the update rule of the algorithm, and $(ii)$ follows from the definition of $d_{\tau}(\cdot)$. Multiplying $(1-\alpha)^{T-1-t}$ on both sides of the above inequality and summing up from $t=1$ to $T-1$, we obtain that
\begin{align}
d_{\tau}(\underline \lambda_{T}) &\leq (1-\alpha)^{T-1} d_{\tau}(\underline \lambda_{1}) + \sum\nolimits_{t=2}^{T} (1-\alpha)^{T-t} [(1-q)\alpha + q] [V_0^{\pi_{t}}(\rho) + \lambda_{t-1}^\top ({V}^{\pi_{t}}(\rho)- c) ]\nonumber\\
&\quad -\sum\nolimits_{t=2}^T (1-\alpha)^{T-t} q(1-\alpha)[V_0^{\pi_{t}}(\rho) + \lambda_{t-2}^\top ({V}^{\pi_{t}}(\rho)- c) ] + \sum\nolimits_{t=2}^{T} (1-\alpha)^{T-t}\Delta_{t}\nonumber\\
&\overset{(i)} \leq (1-\alpha)^{T-1} V_0^{\pi_1} (\rho) + \sum\nolimits_{t=2}^T (1-\alpha)^{T-t} [(1-q)\alpha + q] [V_0^{\pi_{t}}(\rho) + \lambda_{t-1}^\top ({V}^{\pi_{t}}(\rho)- c) ]\nonumber\\
&\quad -\sum\nolimits_{t=2}^T (1-\alpha)^{T-t} q(1-\alpha)[V_0^{\pi_{t}}(\rho) + \lambda_{t-2}^\top ({V}^{\pi_{t}}(\rho)- c) ] + \sum\nolimits_{t=1}^{T} (1-\alpha)^{T-t}\Delta_{t}\nonumber\\
&\leq (1-\alpha)^{T-1} V_0^{\pi_1} (\rho) + \sum\nolimits_{t=2}^T \alpha(1-\alpha)^{T-t} (V_0^{\pi_{t}}(\rho) + \lambda_{t-1}^{\top} ({V}^{\pi_{t}}(\rho)- c)) \nonumber\\
&\quad +\sum\nolimits_{t=2}^T(1-\alpha)^{T-t+1} q(\lambda_{t-1}-\lambda_{t-2})^\top ({V}^{\pi_{t}}(\rho)- c) + \sum\nolimits_{t=1}^{T} (1-\alpha)^{T-t}\Delta_{t},   \label{d_1}
\end{align}
where $(i)$ follows from $\underline\lambda_1=\lambda_0=0$. 

Next, we build the connection between $d_{\tau}(\underline \lambda_{T})$ and the optimal primal function value $V_0^*(\rho)$. From the above inequality and the definition of $d_{\tau}(\cdot)$ and $d_{\tau,\mu}(\cdot)$, it follows that
\begin{align}
d_{\tau}(\underline \lambda_{T}) &= d_{\tau,\mu}(\underline \lambda_{T}) - \tfrac{\mu}{2} \|\underline \lambda_{T}\|_2^2\nonumber\\
&\geq d_{\tau,\mu}(\lambda_{\tau,\mu}^*) - \tfrac{\mu}{2} \|\underline \lambda_{T}\|_2^2\nonumber\\
&= d_{\tau}(\lambda_{\tau,\mu}^*) +\tfrac{\mu}{2} (\|\lambda_{\tau,\mu}^*\|_2^2-\|\underline \lambda_{T}\|_2^2)\nonumber\\
&\overset{(i)}\geq V_0^*(\rho) - \tfrac{\tau}{1-\gamma}\log\left(|\mca|\right)+\tfrac{\mu}{2} (\|\lambda_{\tau,\mu}^*\|_2^2-\|\underline \lambda_{T}\|_2^2),\label{eq:middle1}
\end{align}
where $(i)$ follows from the inductions below 
\begin{align*}
    d_{\tau}(\lambda_{\tau,\mu}^*) &\ge \min_{\lambda\in\mR_+^m} d_{\tau}(\lambda)\\
    & \overset{(i)}= \max_{\pi\in\Pi} \{ V_0^\pi(\rho) + \tau\mcH(\pi): V_i^\pi(\rho)\ge c_i\}\\
    &\overset{(ii)}\ge \max_{\pi\in\Pi} \{ V_0^\pi(\rho): V_i^\pi(\rho)\ge c_i\} -  \tfrac{\tau}{1-\gamma}\log(|\mca|)\\
    &= V_0^*(\rho) -  \tfrac{\tau}{1-\gamma}\log(|\mca|),
\end{align*}
where $(i)$ follows from the zero duality gap of constrained reinforcement learning \citep{paternain2019constrained}, and $(ii)$ follows from \Cref{lemma:maximumfunction} and $|\mcH(\pi)| \le \log(|\mca|)$ for all $\pi$.

Combining \cref{d_1,eq:middle1}, it follows that
\begin{align}
&V_0^*(\rho)  - \left((1-\alpha)^{T-1} V_0^{\pi_1}(\rho)  + \sum\nolimits_{t=2}^T \alpha(1-\alpha)^{T-t} V_0^{\pi_{t}}(\rho) \right) \nonumber\\
\leq&  \sum\nolimits_{t=2}^T \alpha(1-\alpha)^{T-t} \lambda_{t-1}^{\top} ({V}^{\pi_{t}}(\rho)- c)+\sum\nolimits_{t=2}^T(1-\alpha)^{T-t+1} q(\lambda_{t-1}-\lambda_{t-2})^\top ({V}^{\pi_{t}}(\rho)- c) \nonumber\\
&+ \sum\nolimits_{t=1}^{T} (1-\alpha)^{T-t}\Delta_{t}+\tfrac{\tau}{1-\gamma}\log\left(|\mca|\right) +\tfrac{\mu}{2} (\|\underline \lambda_{T}\|_2^2-\|\lambda_{\tau,\mu}^*\|_2^2)\nonumber\\
\leq& \sum\nolimits_{t=2}^T \alpha(1-\alpha)^{T-t} \lambda_{t}^{\top} ({V}^{\pi_{t}}(\rho)- c)+\sum\nolimits_{t=2}^T(1-\alpha)^{T-t+1} q(\lambda_{t-1}-\lambda_{t-2})^\top ({V}^{\pi_{t}}(\rho)- c)     \nonumber\\
&+ \sum\nolimits_{t=2}^T \alpha(1-\alpha)^{T-t} (\lambda_{t-1}-\lambda_{t})^{\top} ({V}^{\pi_{t}(\rho)}- c)+ \sum\nolimits_{t=1}^{T} (1-\alpha)^{T-t}\Delta_{t} + \tfrac{\tau}{1-\gamma}\log\left(|\mca|\right) \nonumber\\
&+\tfrac{\mu}{2} (\|\underline \lambda_{T}\|_2^2-\|\lambda_{\tau,\mu}^*\|_2^2)\nonumber\\
\leq&\sum\nolimits_{t=2}^T \alpha(1-\alpha)^{T-t} \lambda_{t}^{\top} ({V}^{\pi_{t}}(\rho)- c)+ \left(\sum\nolimits_{t=2}^T\alpha(1-\alpha)^{T-t+1} \|\lambda_{t-1}-\lambda_{t-2}\|_2\right)\tfrac{q}{\alpha} \tfrac{R_{\max}}{1-\gamma}   \nonumber\\
&+ \left(\sum\nolimits_{t=2}^T\alpha(1-\alpha)^{T-t} \|\lambda_{t}-\lambda_{t-1}\|_2\right)\tfrac{R_{\max}}{1-\gamma} + \sum\nolimits_{t=1}^{T} (1-\alpha)^{T-t}\Delta_{t} +\tfrac{\tau}{1-\gamma}\log |\mca| +2\mu m B^2.\label{d_3}
\end{align}
In order to analyze the first term of the right hand side of \eqref{d_3}, we exploit the first order optimality condition of step 7 in \Cref{alg:arcpo} to upper bound the term $(\lambda-\lambda_{t})^{\top}( c-{V}^{\pi_{t}}(\rho))$. The first-order optimality condition implies that for $t=1,...,T$, $\forall \lambda\in \Lambda$,
\begin{align}
0 &\leq \langle \eta({V}^{\pi_{t}}(\rho)- c+\mu \underline \lambda_t)+\eta \mu (\lambda_t- \underline \lambda_t)+ (\lambda_t- \lambda_{t-1}), \lambda-\lambda_t\rangle \nonumber\\
&=\langle \eta({V}^{\pi_{t}}(\rho)- c)+(\eta\mu+1) \lambda_t- \lambda_{t-1}, \lambda-\lambda_t\rangle \nonumber\\
&=\langle (\eta\mu+1) \lambda_t- \lambda_{t-1}, \lambda-\lambda_t\rangle + \eta(\lambda-\lambda_{t})^{\top}({V}^{\pi_{t}}(\rho)- c)\nonumber\\
&\leq \|\lambda_{t-1}-(\eta\mu+1)\lambda_{t}\|_2\|\lambda-\lambda_t\|_2+ \eta(\lambda-\lambda_{t})^{\top}({V}^{\pi_{t}}(\rho)- c)\nonumber\\
&\leq 2\|\lambda_t-\lambda_{t-1}\|_2\sqrt{m}B+4\eta\mu mB^2+ \eta(\lambda-\lambda_{t})^{\top}({V}^{\pi_{t}}(\rho)- c).\label{d_4}
\end{align}
Adding both sides of \cref{d_3,d_4}, we have that
\begin{align}
&V_0^*(\rho)  - \left((1-\alpha)^{T-1} V_0^{\pi_1}(\rho)  + \sum\nolimits_{t=2}^T \alpha(1-\alpha)^{T-t} (V_0^{\pi_{t}}(\rho) + \lambda^{\top} ( c- {V}^{\pi_{t}}(\rho)))\right)  \nonumber\\
\leq& \left(\sum\nolimits_{t=1}^{T}\alpha(1-\alpha)^{T-t} \|\lambda_t-\lambda_{t-1}\|_2\right)\left(\tfrac{q}{\alpha}\tfrac{R_{\max}}{1-\gamma} + \tfrac{R_{\max}}{1-\gamma} + \tfrac{2\sqrt{m}B}{\eta}\right)\nonumber\\
&+ \sum\nolimits_{t=1}^{T} (1-\alpha)^{T-t}\Delta_{t} +\tfrac{\tau}{1-\gamma}\log |\mca| +6\mu m B^2\nonumber\\
\overset{(i)}\leq& \left(\sum\nolimits_{t=1}^{T} \alpha(1-\alpha)^{T-t}\right)^{\tfrac{1}{2}}\left(\sum\nolimits_{t=1}^{T} \alpha(1-\alpha)^{T-t}\|\lambda_t-\lambda_{t-1}\|_2^2\right)^{\tfrac{1}{2}}\left(\tfrac{q}{\alpha}\tfrac{R_{\max}}{1-\gamma} + \tfrac{R_{\max}}{1-\gamma} + \tfrac{2\sqrt{m}B}{\eta}\right)\nonumber\\
&+ \sum\nolimits_{t=1}^{T} (1-\alpha)^{T-t}\Delta_{t} +\tfrac{\tau}{1-\gamma} \log |\mca| +6\mu m B^2\nonumber\\
\leq& \left(\sum\nolimits_{t=1}^{T} \alpha(1-\alpha)^{T-t}\|\lambda_t-\lambda_{t-1}\|_2^2\right)^{\tfrac{1}{2}}\left(\tfrac{q}{\alpha}\tfrac{R_{\max}}{1-\gamma} + \tfrac{R_{\max}}{1-\gamma} + \tfrac{2\sqrt{m}B}{\eta}\right)\nonumber\\
&+ \sum\nolimits_{t=1}^{T} (1-\alpha)^{T-t}\Delta_{t} +\tfrac{\tau}{1-\gamma} \log |\mca| +6\mu m B^2, \nonumber
\end{align}
where $(i)$ follows from H{\"o}lder's inequality. This yields the desired result as we argue above, and concludes the proof.

\subsubsection{Proof of \Cref{prop:keyconvergenceprop}}  
    
	Denote $g_t = \bar \lambda_t - \underline \lambda_t$. It follows from lines 4, 7 and 8 of \Cref{alg:arcpo} that 
	\begin{align}
	g_t &= (q-\alpha) \bar \lambda_{t-1} + \alpha \lambda_t - q \lambda_{t-1}\nonumber\\
	&=\alpha[\lambda_t - \tfrac{\alpha-q}{\alpha(1-q)}\underline \lambda_t - \tfrac{q(1-\alpha)}{\alpha(1-q)}\lambda_{t-1}].\label{g_t}
	\end{align}
	By the convexity of $d_{\tau, \mu}(\cdot)$, we have that
	\begin{align*}
	d_{\tau, \mu}(\bar \lambda_t) & \leq d_{\tau, \mu}(\underline \lambda_t) + \langle \nabla d_{\tau,\mu}(\underline{\lambda}_t), \bar \lambda_t -\underline\lambda_t \rangle+  \tfrac{L_d}{2}\|g_t\|_2^2\\
	& = (1-\alpha) [d_{\tau, \mu}(\underline \lambda_t)+ \langle \nabla d_{\tau,\mu}(\underline{\lambda}_t), \bar \lambda_{t-1} -\underline\lambda_t \rangle] + \alpha [d_{\tau, \mu}(\underline \lambda_t)+ \langle \nabla d_{\tau,\mu}(\underline{\lambda}_t), \lambda_t-\underline\lambda_t \rangle] \\
	&\quad+ \tfrac{L_d}{2}\|g_t\|_2^2\\
	&\overset{(i)} \leq (1-\alpha)d_{\tau, \mu}(\bar \lambda_{t-1}) + \alpha[d_{\tau, \mu}(\underline \lambda_{t}) + \langle \nabla d_{\tau,\mu}(\underline{\lambda}_t), \lambda_t-\underline\lambda_t \rangle + \tfrac{L_d(a-q)}{2(1-q)}\|\lambda_t - \underline \lambda_t\|_2^2
	\\
	&\quad+\tfrac{L_dq(1-\alpha)}{2(1-q)}\|\lambda_t-\lambda_{t-1}\|_2^2]\\
	&\overset{(ii)} \leq (1-\alpha)d_{\tau, \mu}(\bar \lambda_{t-1}) + \alpha[d_{\tau, \mu}(\underline \lambda_{t}) + \langle \nabla d_{\tau,\mu}(\underline{\lambda}_t), \lambda_t-\underline\lambda_t \rangle + \tfrac{\mu}{2}\|\lambda_t - \underline \lambda_t\|_2^2\\
	&\quad+\tfrac{1}{4\eta}\|\lambda_t-\lambda_{t-1}\|_2^2],
	\end{align*}
	where $(i)$ follows from \cref{g_t}, the convexity of $d_{\tau, \mu}(\cdot)$, the convexity of $\|\cdot\|_2^2$ and the fact that $\alpha \geq q$, and $(ii)$ follows from the fact that $\tfrac{L_d(\alpha-q)}{1-q}\leq \mu$ and $\tfrac{L_dq(1-\alpha)}{1-q}\leq \tfrac{1}{2\eta}$. Combining the above inequality and \cref{3p}, we obtain that
	\begin{align*}
	d_{\tau, \mu}(\bar \lambda_t) & \leq (1-\alpha)d_{\tau, \mu}(\bar \lambda_{t-1}) + \alpha[d_{\tau, \mu}(\underline \lambda_{t}) + \langle \nabla d_{\tau,\mu}(\underline{\lambda}_t), \lambda-\underline\lambda_t \rangle + \tfrac{\mu}{2}\|\lambda - \underline \lambda_t\|_2^2]\\
	&\quad+\tfrac{\alpha}{2\eta}\|\lambda-\lambda_{t-1}\|_2^2 - \tfrac{\alpha}{2}(\mu+\tfrac{1}{\eta})\|\lambda-\lambda_t\|_2^2-\tfrac{\alpha}{4\eta}\|\lambda_t-\lambda_{t-1}\|_2^2+\alpha\langle \delta_t, \lambda - \lambda_t \rangle\\
	& \leq (1-\alpha)d_{\tau, \mu}(\bar \lambda_{t-1}) + \alpha d_{\tau, \mu}(\lambda)+\tfrac{\alpha}{2\eta}\|\lambda-\lambda_{t-1}\|_2^2- \tfrac{\alpha}{2}(\mu+\tfrac{1}{\eta})\|\lambda-\lambda_t\|_2^2\\
	&\quad -\tfrac{\alpha}{4\eta}\|\lambda_t-\lambda_{t-1}\|_2^2+\alpha\langle \delta_t, \lambda - \lambda_t \rangle.
	\end{align*}
	Subtracting $d_{\tau, \mu}(\lambda)$ on both sides of the above inequality and rearranging the terms, we have 
	\begin{align*}
	K_t(\lambda)+ \tfrac{\alpha}{4\eta}\|\lambda_t-\lambda_{t-1}\|_2^2 &\leq (1-\alpha)[d_{\tau, \mu}(\bar \lambda_{t-1})-d_{\tau, \mu}(\lambda)+\tfrac{\alpha}{2\eta(1-\alpha)}\|\lambda_{t-1}-\lambda\|_2^2]\\
	&\quad +\alpha\langle \delta_t, \lambda - \lambda_t \rangle\\
	& \overset{(i)}\leq (1-\alpha)K_{t-1}(\lambda) + \alpha\langle \delta_t, \lambda - \lambda_t \rangle,
	\end{align*}
	where $(i)$ follows from the fact that $\tfrac{1}{\eta (1-\alpha)} \leq \mu + \tfrac{1}{\eta}$. The desired result follows directly by multiplying $(1-\alpha)^{T-t}$ on both sides of the above inequality and summing up from $t=1$ to $T$. 

\subsubsection{Proof of \Cref{prop:constraintviolation}} 
We first define the perturbation function, given $\zeta\in \mR^m$, $f(\zeta) = \max\{V_0^{\pi}(\rho)| V_{i}^{\pi}(\rho) \ge c_i + \zeta_i\}$. It was shown in \cite[Proposition 1]{paternain2019safe} that $f(\zeta)$ is  concave in $\zeta$. 

For a given $\zeta\in\mR^m$ and any policy  $\pi \in\left\{\pi\in\Pi| V_i^{\pi}(\rho) \ge c_i + \zeta_i\right\}$, we have that
\begin{align*}
    f(0) - \langle \zeta, \lambda^*\rangle &\overset{(i)}= \mcL(\pi^*, \lambda^*) - \langle \zeta, \lambda^*\rangle \\
    &\ge \mcL(\pi, \lambda^*) - \langle \zeta, \lambda^*\rangle \\
    &= V_0^\pi(\rho) + \langle V^{\pi}(\rho) -\cb, \lambda^*\rangle \\
    &\ge V_0^\pi (\rho),
\end{align*}
where $(i)$ follows from the fact that $f(0) = V_0^*(\rho) = \mcL(\pi^*, \lambda^*)$.
The above inequality implies that 
\begin{align*}
    f(0) - \langle \zeta, \lambda^*\rangle  \ge f(\zeta).
\end{align*}
Moreover, using above inequality, we have 
\begin{align*}
    V_0^{\tilde{\pi}}(\rho)\le f(-(\cb - V^{\tilde\pi}(\rho))_+) \le V_0^{\pi^*} + \langle (\cb - V^{\tilde\pi}(\rho))_+, \lambda^*\rangle. 
\end{align*}
Combing the above inequality and the bound given in the proposition, i.e.\[V_0^*(\rho) - V_0^{\tilde{\pi}}(\rho)+ \langle \lambda, (\cb - V^{\tilde\pi}(\rho))_+ \rangle\le\sigma,\] we obtain
\begin{equation*}
    \langle\lambda - \lambda^*, (\cb - V^{\tilde\pi}(\rho))_+ \rangle \le \sigma.
\end{equation*}
Using the condition that $\lambda\succeq \lambda^*+B\mathbf{1}$, we complete the proof.

\subsection{Proof of \Cref{corollary:totalcomplexity}}\label{sec:proofofcorollary1}
We pick $\tau \leq {\epsilon}/{\log (|\mca|)}$, $\mu \leq {\epsilon}/{(6m B^2)}$,
    \begin{align}\label{eq:deltaorder}
    \delta\le \max \bigg\{\epsilon^2/\big(8L_\nu R_{\max}( \tfrac{2R_{\max}}{1-\gamma} + \tfrac{2\sqrt{m}B}{\eta})^2\eta\sqrt{m}B\big), \tfrac{\epsilon}{L_\nu r_{0,\max}}\sqrt{\tfrac{\mu}{2L_d}}\bigg\},
    \end{align}
    and 
    \begin{align}\label{eq:Torder}
    T\geq \sqrt{\tfrac{2 L_d}{\mu}}\max\bigg\{2\log\bigg(2\sqrt{\eta K_0(\lambda^*_{\tau,\mu})}\bigg(\tfrac{2R_{\max}}{1-\gamma} + \tfrac{2\sqrt{m}B}{\eta}\bigg)\tfrac{1}{\epsilon}\bigg), \log\bigg(\tfrac{2e\sqrt{m}B R_{\max}}{(1-\gamma)\epsilon}\bigg)\bigg\}.
    \end{align}
    
    By the design of $\textrm{RegPO}$ (see \Cref{app:regpo}), we have that for all $t=1,\ldots T$, $\|\delta_t\|_2\le L_{\nu}R_{\max}\delta$ and $|\Delta_t|\le L_{\nu}r_{0,\max}\delta$, and thus, following from Theorem \ref{thm1}, $V_0^*(\rho)-V^{\tilde \pi}_0(\rho)\leq 5\epsilon$ and $\|({c}-{V}^{\tilde \pi}(\rho))_+\|_1\leq {6\epsilon}/{B}$.
    
To obtain the overall complexity, following our choices $\tau=\mathcal{O}(\epsilon)$ and $\mu=\mathcal{O}(\epsilon)$, the Lipschitz constant $L_d= \tfrac{2R_{\max}^2L_\nu}{(1-\gamma)^2\tau} +\mu=\mathcal{O}({1}/{\epsilon})$ (see \Cref{prop:smoothness}), and then $\delta=\mathcal{O}\left(\epsilon^2\right)$ due to \cref{eq:deltaorder} and $T= \mathcal{O}({1}/{\epsilon}\log({1}/{\epsilon}))$ due to \cref{eq:Torder}. Hence, $\textrm{RegPO}$ $(\delta,\lambda,\tau)$ takes $K = \mathcal{O}(\log(1/(\delta\tau))) = \mathcal{O}(\log({1}/{\epsilon}))$ iterations (see \Cref{app:regpo}). Therefore, the overall computational complexity of Algorithm \ref{alg:arcpo} is given by $KT =\mathcal O({1}/{\epsilon}\log^2({1}/{\epsilon}))$.

\end{document}